\documentclass[12pt]{amsart}

\usepackage{amsmath}
\usepackage{amssymb}
\usepackage{enumerate}
\usepackage{graphicx}

\makeatletter
\@namedef{subjclassname@2010}{%
  \textup{2010} Mathematics Subject Classification}
\makeatother

\newtheorem{thm}{Theorem}[section]

\newtheorem{lem}[thm]{Lemma}
\newtheorem{prop}[thm]{Proposition}

\theoremstyle{definition}
\newtheorem{defin}[thm]{Definition}

\numberwithin{equation}{section}

\begin{document}


\baselineskip=17pt



\title[Weighted Morrey spaces related to nonnegative potentials]{Weighted Morrey spaces related to certain nonnegative potentials and Riesz transforms}

\author[H. Wang]{Hua Wang}
\address{College of Mathematics and Econometrics, Hunan University, Changsha, 410082, P. R. China\\
\&~Department of Mathematics and Statistics, Memorial University, St. John's, NL A1C 5S7, Canada}
\email{wanghua@pku.edu.cn}
\date{}

\begin{abstract}
Let $\mathcal L=-\Delta+V$ be a Schr\"odinger operator, where $\Delta$ is the Laplacian on $\mathbb R^d$ and the nonnegative potential $V$ belongs to the reverse H\"older class $RH_q$ for $q\geq d$. The Riesz transform associated with the operator $\mathcal L=-\Delta+V$ is denoted by $\mathcal R=\nabla{(-\Delta+V)}^{-1/2}$ and the dual Riesz transform is denoted by $\mathcal R^{\ast}=(-\Delta+V)^{-1/2}\nabla$. In this paper, we first introduce some kinds of weighted Morrey spaces related to certain nonnegative potentials belonging to the reverse H\"older class $RH_q$ for $q\geq d$. Then we will establish the boundedness properties of the operators $\mathcal R$ and its adjoint $\mathcal R^{\ast}$ on these new spaces. Furthermore, weighted strong-type estimate and weighted endpoint estimate for the corresponding commutators $[b,\mathcal R]$ and $[b,\mathcal R^{\ast}]$ are also obtained. The classes of weights, the classes of symbol functions as well as weighted Morrey spaces discussed in this paper are larger than $A_p$, $\mathrm{BMO}(\mathbb R^d)$ and $L^{p,\kappa}(w)$ corresponding to the classical Riesz transforms ($V\equiv0$).
\end{abstract}

\maketitle

\section{Introduction}

\subsection{The critical radius function $\rho(x)$}
Let $d\geq3$ be a positive integer and $\mathbb R^d$ be the $d$-dimensional Euclidean space. A nonnegative locally integrable function $V(x)$ on $\mathbb R^d$ is said to belong to the reverse H\"older class $RH_q$ for some exponent $1<q<\infty$, if there exists a positive constant $C>0$ such that the following reverse H\"older inequality
\begin{equation*}
\left(\frac{1}{|B|}\int_B V(y)^q\,dy\right)^{1/q}\leq C\left(\frac{1}{|B|}\int_B V(y)\,dy\right)
\end{equation*}
holds for every ball $B$ in $\mathbb R^d$.
For given $V\in RH_q$ with $q\geq d$, we introduce the \emph{critical radius function} $\rho(x)=\rho(x,V)$ which is given by
\begin{equation}\label{rho}
\rho(x):=\sup_{r>0}\bigg\{r:\frac{1}{r^{d-2}}\int_{B(x,r)}V(y)\,dy\leq1\bigg\},\quad x\in\mathbb R^d,
\end{equation}
where $B(x,r)$ denotes the open ball centered at $x$ and with radius $r$. It is well known that $0<\rho(x)<\infty$ for any $x\in\mathbb R^d$ under our assumption (see \cite{shen}). We need the following known result concerning the critical radius function.
\begin{lem}[\cite{shen}]\label{N0}
If $V\in RH_q$ with $q\geq d$, then there exist two constants $C>0$ and $N_0\geq 1$ such that
\begin{equation}\label{com}
\frac{\,1\,}{C}\left(1+\frac{|x-y|}{\rho(x)}\right)^{-N_0}\leq\frac{\rho(y)}{\rho(x)}\leq C\left(1+\frac{|x-y|}{\rho(x)}\right)^{\frac{N_0}{N_0+1}},
\end{equation}
for all $x,y\in\mathbb R^d$. As a straightforward consequence of \eqref{com}, we have that for all $k=1,2,3,\dots$, the following estimate
\begin{equation}\label{com2}
1+\frac{2^kr}{\rho(y)}\geq C\left(1+\frac{r}{\rho(x)}\right)^{-\frac{N_0}{N_0+1}}\left(1+\frac{2^kr}{\rho(x)}\right)
\end{equation}
is valid for any $y\in B(x,r)$ with $x\in\mathbb R^d$ and $r>0$.
\end{lem}

\subsection{Schr\"odinger operators}
On $\mathbb R^d$, $d\geq3$, we consider the \emph{Schr\"odinger operator}
\begin{equation*}
\mathcal L:=-\Delta+V,
\end{equation*}
where $V\in RH_q$ for $q\geq d$. The \emph{Riesz transform} associated with the Schr\"odinger operator $\mathcal L$ is defined by
\begin{equation}\label{riesz}
\mathcal R:=\nabla{\mathcal L}^{-1/2},
\end{equation}
and the associated \emph{dual Riesz transform} is defined by
\begin{equation}\label{riesz2}
\mathcal R^{\ast}:={\mathcal L}^{-1/2}\nabla.
\end{equation}
Boundedness properties of $\mathcal R$ and its adjoint $\mathcal R^{\ast}$ have been obtained by Shen in \cite{shen}, where he showed that they are all bounded on $L^p(\mathbb R^d)$ for any $1<p<\infty$ when $V\in RH_q$ with $q\geq d$. Actually, $\mathcal R$ and its adjoint $\mathcal R^{\ast}$ are standard Calder\'on-Zygmund operators in such a situation. The operators $\mathcal R$ and $\mathcal R^{\ast}$ have singular kernels that will be denoted by $\mathcal K(x,y)$ and $\mathcal K^{\ast}(x,y)$, respectively. For such kernels, we have the following key estimates, which can be found in \cite{shen} and \cite{bong2,bong3}.
\begin{lem}\label{kernel}
Let $V\in RH_q$ with $q\geq d$. For any positive integer $N$, there exists a positive constant $C_N>0$ such that
\begin{equation*}
\begin{cases}
\mathcal |\mathcal K(x,y)|\leq C_N\displaystyle\bigg(1+\frac{|x-y|}{\rho(x)}\bigg)^{-N}\frac{1}{|x-y|^d};&\\
\mathcal |\mathcal K^{\ast}(x,y)|\leq C_N\displaystyle\bigg(1+\frac{|x-y|}{\rho(x)}\bigg)^{-N}\frac{1}{|x-y|^d}.&
\end{cases}
\end{equation*}
\end{lem}

\subsection{$A^{\rho,\infty}_p$ weights}
A weight will always mean a nonnegative function which is locally integrable on $\mathbb R^d$. Given a Lebesgue measurable set $E$ and a weight $w$, $|E|$ will denote the Lebesgue measure of $E$ and
\begin{equation*}
w(E)=\int_E w(x)\,dx.
\end{equation*}
Given $B=B(x_0,r)$ and $t>0$, we will write $tB$ for the $t$-dilate ball, which is the ball with the same center $x_0$ and with radius $tr$. In \cite{bong1} (see also \cite{bong2,bong3}), Bongioanni, Harboure and Salinas introduced the following classes of weights that are given in terms of the critical radius function \eqref{rho}. Following the terminology of \cite{bong1}, for given $1<p<\infty$, we define
\begin{equation*}
A^{\rho,\infty}_p:=\bigcup_{\theta>0}A^{\rho,\theta}_p,
\end{equation*}
where $A^{\rho,\theta}_p$ is the set of all weights $w$ such that
\begin{equation*}
\bigg(\frac{1}{|B|}\int_B w(x)\,dx\bigg)^{1/p}\bigg(\frac{1}{|B|}\int_B w(x)^{-{p'}/p}\,dx\bigg)^{1/{p'}}
\leq C\cdot\left(1+\frac{r}{\rho(x_0)}\right)^{\theta}
\end{equation*}
holds for every ball $B=B(x_0,r)\subset\mathbb R^d$ with $x_0\in\mathbb R^d$ and $r>0$, where $p'$ is the dual exponent of $p$ such that $1/p+1/{p'}=1$. For $p=1$ we define
\begin{equation*}
A^{\rho,\infty}_1:=\bigcup_{\theta>0}A^{\rho,\theta}_1,
\end{equation*}
where $A^{\rho,\theta}_1$ is the set of all weights $w$ such that
\begin{equation*}
\frac1{|B|}\int_B w(x)\,dx\leq C\cdot\left(1+\frac{r}{\rho(x_0)}\right)^{\theta}\underset{x\in B}{\mbox{ess\,inf}}\;w(x)
\end{equation*}
holds for every ball $B=B(x_0,r)$ in $\mathbb R^d$. For $\theta>0$, let us introduce the maximal operator that is given in terms of the critical radius function \eqref{rho}.
\begin{equation*}
M_{\rho,\theta}f(x):=\sup_{r>0}\left(1+\frac{r}{\rho(x)}\right)^{-\theta}\frac{1}{|B(x,r)|}\int_{B(x,r)}|f(y)|\,dy,\quad x\in\mathbb R^d.
\end{equation*}
Observe that a weight $w$ belongs to the class $A^{\rho,\infty}_1$ if and only if there exists a positive number $\theta>0$ such that $M_{\rho,\theta}w\leq Cw$, where the constant $C>0$ is independent of $w$. Since
\begin{equation*}
1\leq\left(1+\frac{r}{\rho(x_0)}\right)^{\theta_1}\leq\left(1+\frac{r}{\rho(x_0)}\right)^{\theta_2}
\end{equation*}
for $0<\theta_1<\theta_2<\infty$, then for given $p$ with $1\leq p<\infty$, one has
\begin{equation*}
A_p\subset A^{\rho,\theta_1}_p\subset A^{\rho,\theta_2}_p,
\end{equation*}
where $A_p$ denotes the classical Muckenhoupt's class (see \cite[Chapter 7]{grafakos}), and hence $A_p\subset A^{\rho,\infty}_p$. In addition, for some fixed $\theta>0$,
\begin{equation*}
A^{\rho,\theta}_1\subset A^{\rho,\theta}_{p_1}\subset A^{\rho,\theta}_{p_2}
\end{equation*}
whenever $1\leq p_1<p_2<\infty$. Now, we present an important property of the classes of weights in $A^{\rho,\theta}_p$ with $1\leq p<\infty$, which was given by Bongioanni et al. in \cite[Lemma 5]{bong1}.
\begin{lem}[\cite{bong1}]\label{rh}
If $w\in A^{\rho,\theta}_p$ with $0<\theta<\infty$ and $1\leq p<\infty$, then there exist positive constants $\epsilon,\eta>0$ and $C>0$ such that
\begin{equation}\label{rholder}
\left(\frac{1}{|B|}\int_B w(x)^{1+\epsilon}dx\right)^{\frac{1}{1+\epsilon}}
\leq C\left(\frac{1}{|B|}\int_B w(x)\,dx\right)\left(1+\frac{r}{\rho(x_0)}\right)^{\eta}
\end{equation}
for every ball $B=B(x_0,r)$ in $\mathbb R^d$.
\end{lem}

As a direct consequence of Lemma \ref{rh}, we have the following result.
\begin{lem}\label{comparelem}
If $w\in A^{\rho,\theta}_p$ with $0<\theta<\infty$ and $1\leq p<\infty$, then there exist two positive numbers $\delta>0$ and $\eta>0$ such that
\begin{equation}\label{compare}
\frac{w(E)}{w(B)}\leq C\left(\frac{|E|}{|B|}\right)^\delta\left(1+\frac{r}{\rho(x_0)}\right)^{\eta}
\end{equation}
for any measurable subset $E$ of a ball $B=B(x_0,r)$, where $C>0$ is a constant which does not depend on $E$ and $B$.
\end{lem}
For any given ball $B=B(x_0,r)$ with $x_0\in\mathbb R^d$ and $r>0$, suppose that $E\subset B$, then by H\"older's inequality with exponent $1+\epsilon$ and \eqref{rholder}, we can deduce that
\begin{equation*}
\begin{split}
w(E)&=\int_{B}\chi_E(x)\cdot w(x)\,dx\\
&\leq
\left(\int_B w(x)^{1+\epsilon}dx\right)^{\frac{1}{1+\epsilon}}
\left(\int_B\chi_E(x)^{\frac{1+\epsilon}{\epsilon}}\,dx\right)^{\frac{\epsilon}{1+\epsilon}}\\
&\leq C|B|^{\frac{1}{1+\epsilon}}\left(\frac{1}{|B|}\int_B w(x)\,dx\right)\left(1+\frac{r}{\rho(x_0)}\right)^{\eta}|E|^{\frac{\epsilon}{1+\epsilon}}\\
&=C\left(\frac{|E|}{|B|}\right)^{\frac{\epsilon}{1+\epsilon}}\left(1+\frac{r}{\rho(x_0)}\right)^{\eta}.
\end{split}
\end{equation*}
This gives \eqref{compare} with $\delta=\epsilon/{(1+\epsilon)}$.

Given a weight $w$ on $\mathbb R^d$, as usual, the weighted Lebesgue space $L^p(w)$ for $1\leq p<\infty$ is defined to be the set of all functions $f$ such that
\begin{equation*}
\big\|f\big\|_{L^p(w)}:=\bigg(\int_{\mathbb R^d}\big|f(x)\big|^pw(x)\,dx\bigg)^{1/p}<\infty.
\end{equation*}
We also denote by $WL^1(w)$ the weighted weak Lebesgue space consisting of all measurable functions $f$ for which
\begin{equation*}
\big\|f\big\|_{WL^1(w)}:=
\sup_{\lambda>0}\lambda\cdot w\big(\big\{x\in\mathbb R^d:|f(x)|>\lambda\big\}\big)<\infty.
\end{equation*}

Recently, Bongioanni et al.\cite{bong1} obtained weighted strong-type and weak-type estimates for the operators $\mathcal R$ and $\mathcal R^{\ast}$ defined in \eqref{riesz} and \eqref{riesz2}. Their results can be summarized as follows:

\begin{thm}[\cite{bong1}]\label{strong}
Let $1<p<\infty$ and $w\in A^{\rho,\infty}_p$. If $V\in RH_q$ with $q\geq d$, then the operators $\mathcal R$ and $\mathcal R^{\ast}$ are all bounded on $L^p(w)$.
\end{thm}

\begin{thm}[\cite{bong1}]\label{weak}
Let $p=1$ and $w\in A^{\rho,\infty}_1$. If $V\in RH_q$ with $q\geq d$, then the operators $\mathcal R$ and $\mathcal R^{\ast}$ are all bounded from $L^1(w)$ into $WL^1(w)$.
\end{thm}

\subsection{The space $\mathrm{BMO}_{\rho,\infty}(\mathbb R^d)$}
We denote by $\mathcal T$ either $\mathcal R$ or $\mathcal R^{\ast}$. For a locally integrable function $b$ on $\mathbb R^d$ (usually called the \emph{symbol}), we will also consider the commutator operator
\begin{equation}\label{briesz}
[b,\mathcal T]f(x):=b(x)\cdot \mathcal Tf(x)-\mathcal T(bf)(x),\quad x\in\mathbb R^d.
\end{equation}
Recently, Bongioanni et al. \cite{bong3} introduced a new space $\mathrm{BMO}_{\rho,\infty}(\mathbb R^d)$ defined by
\begin{equation*}
\mathrm{BMO}_{\rho,\infty}(\mathbb R^d):=\bigcup_{\theta>0}\mathrm{BMO}_{\rho,\theta}(\mathbb R^d),
\end{equation*}
where for $0<\theta<\infty$ the space $\mathrm{BMO}_{\rho,\theta}(\mathbb R^d)$ is defined to be the set of all locally integrable functions $b$ satisfying
\begin{equation}\label{BM}
\frac{1}{|B(x_0,r)|}\int_{B(x_0,r)}\big|b(x)-b_{B(x_0,r)}\big|\,dx
\leq C\cdot\left(1+\frac{r}{\rho(x_0)}\right)^{\theta},
\end{equation}
for all $x_0\in\mathbb R^d$ and $r>0$, $b_{B(x_0,r)}$ denotes the mean value of $b$ on $B(x_0,r)$, that is,
\begin{equation*}
b_{B(x_0,r)}:=\frac{1}{|B(x_0,r)|}\int_{B(x_0,r)}b(y)\,dy.
\end{equation*}
A norm for $b\in \mathrm{BMO}_{\rho,\theta}(\mathbb R^d)$, denoted by $\|b\|_{\mathrm{BMO}_{\rho,\theta}}$, is given by the infimum of the constants satisfying \eqref{BM}, or equivalently,
\begin{equation*}
\|b\|_{\mathrm{BMO}_{\rho,\theta}}
:=\sup_{B(x_0,r)}\left(1+\frac{r}{\rho(x_0)}\right)^{-\theta}\bigg(\frac{1}{|B(x_0,r)|}\int_{B(x_0,r)}\big|b(x)-b_{B(x_0,r)}\big|\,dx\bigg),
\end{equation*}
where the supremum is taken over all balls $B(x_0,r)$ with $x_0\in\mathbb R^d$ and $r>0$.
With the above definition in mind, one has
\begin{equation*}
\mathrm{BMO}(\mathbb R^d)\subset \mathrm{BMO}_{\rho,\theta_1}(\mathbb R^d)\subset \mathrm{BMO}_{\rho,\theta_2}(\mathbb R^d)
\end{equation*}
for $0<\theta_1<\theta_2<\infty$, and hence $\mathrm{BMO}(\mathbb R^d)\subset\mathrm{BMO}_{\rho,\infty}(\mathbb R^d)$. Moreover, the classical BMO space \cite{john} is properly contained in $\mathrm{BMO}_{\rho,\infty}(\mathbb R^d)$ (see \cite{bong2,bong3} for some examples).
We need the following key result for $\mathrm{BMO}_{\rho,\theta}(\mathbb R^d)$, which was proved by Tang in \cite{tang}.
\begin{prop}[\cite{tang}]\label{tangbmo}
Let $b\in\mathrm{BMO}_{\rho,\theta}(\mathbb R^d)$ with $0<\theta<\infty$. Then there exist two positive constants $C_1$ and $C_2$ such that for any given ball $B(x_0,r)$ in $\mathbb R^d$ and for any $\lambda>0$, we have
\begin{equation}\label{tang}
\begin{split}
&\big|\big\{x\in B(x_0,r):|b(x)-b_{B(x_0,r)}|>\lambda\big\}\big|\\
&\leq C_1|B(x_0,r)|
\exp\bigg[-\bigg(1+\frac{r}{\rho(x_0)}\bigg)^{-\theta^{\ast}}\frac{C_2 \lambda}{\|b\|_{\mathrm{BMO}_{\rho,\theta}}}\bigg],
\end{split}
\end{equation}
where $\theta^{\ast}=(N_0+1)\theta$ and $N_0$ is the constant appearing in Lemma \ref{N0}.
\end{prop}
As a consequence of Proposition \ref{tangbmo} and Lemma \ref{comparelem}, we have the following result:
\begin{prop}\label{wangbmo}
Let $b\in\mathrm{BMO}_{\rho,\theta}(\mathbb R^d)$ with $0<\theta<\infty$ and $w\in A^{\rho,\infty}_p$ with $1\leq p<\infty$. Then there exist positive constants $C_1,C_2$ and $\eta>0$ such that for any given ball $B(x_0,r)$ in $\mathbb R^d$ and for any $\lambda>0$, we have
\begin{equation}\label{wang}
\begin{split}
&w\big(\big\{x\in B(x_0,r):|b(x)-b_{B(x_0,r)}|>\lambda\big\}\big)\\
&\leq C_1w\big(B(x_0,r)\big)
\exp\bigg[-\bigg(1+\frac{r}{\rho(x_0)}\bigg)^{-\theta^{\ast}}\frac{C_2 \lambda}{\|b\|_{\mathrm{BMO}_{\rho,\theta}}}\bigg]
\left(1+\frac{r}{\rho(x_0)}\right)^{\eta},
\end{split}
\end{equation}
where $\theta^{\ast}=(N_0+1)\theta$ and $N_0$ is the constant appearing in Lemma \ref{N0}.
\end{prop}

\subsection{Orlicz spaces}
In this subsection, let us give the definition and some basic facts about Orlicz spaces. For more information on this subject, the reader may consult the book \cite{rao}. Recall that a function $\mathcal A:[0,\infty)\rightarrow[0,\infty)$ is called a Young function if it is continuous, convex and strictly increasing with
\begin{equation*}
\mathcal A(0)=0\qquad  \& \qquad \lim_{t\to\infty}\mathcal A(t)\to\infty.
\end{equation*}
An important example of Young function is $\mathcal A(t)=t\cdot(1+\log^+t)^m$ with some $1\leq m<\infty$. Given a Young function $\mathcal A$ and a function $f$ defined on a ball $B$, we consider the $\mathcal A$-average of a function $f$ given by the following Luxemburg norm:
\begin{equation*}
\big\|f\big\|_{\mathcal A,B}
:=\inf\left\{\lambda>0:\frac{1}{|B|}\int_B\mathcal A\left(\frac{|f(x)|}{\lambda}\right)dx\leq1\right\}.
\end{equation*}
Associated to each Young function $\mathcal A$, one can define its complementary function $\bar{\mathcal A}$ as follows:
\begin{equation*}
\bar{\mathcal A}(s):=\sup_{t>0}\big\{st-\mathcal A(t)\big\}.
\end{equation*}
Such a function $\bar{\mathcal A}$ is also a Young function. It is well known that the following \emph{generalized H\"older inequality} in Orlicz spaces holds for any given ball $B\subset\mathbb R^d$:
\begin{equation*}
\frac{1}{|B|}\int_B\big|f(x)\cdot g(x)\big|\,dx\leq 2\big\|f\big\|_{\mathcal A,B}\big\|g\big\|_{\bar{\mathcal A},B}.
\end{equation*}
In particular, for the Young function $\mathcal A(t)=t\cdot(1+\log^+t)$, the Luxemburg norm will be denoted by $\|\cdot\|_{L\log L,B}=\|\cdot\|_{\mathcal A,B}$. A simple computation shows that the complementary Young function of $\mathcal A(t)=t\cdot(1+\log^+t)$ is $\bar{\mathcal A}(t)=\exp(t)-1$. The corresponding Luxemburg norm will be denoted by $\|\cdot\|_{\exp L,B}=\|\cdot\|_{\bar{\mathcal A},B}$. In this situation, we have
\begin{equation}\label{holder}
\frac{1}{|B|}\int_B\big|f(x)\cdot g(x)\big|\,dx\leq 2\big\|f\big\|_{L\log L,B}\big\|g\big\|_{\exp L,B}.
\end{equation}
We next define the weighted $\mathcal A$-average of a function $f$ over a ball $B$. Given a Young function $\mathcal A$ and a weight function $w$, let (see \cite{rao} for instance)
\begin{equation*}
\big\|f\big\|_{\mathcal A(w),B}:=\inf\left\{\lambda>0:\frac{1}{w(B)}
\int_B\mathcal A\left(\frac{|f(x)|}{\lambda}\right)\cdot w(x)\,dx\leq1\right\}.
\end{equation*}
When $\mathcal A(t)=t$, we denote $\|\cdot\|_{L(w),B}=\|\cdot\|_{\mathcal A(w),B}$, and when $\Phi(t)=t\cdot(1+\log^+t)$, we denote $\|\cdot\|_{L\log L(w),B}=\|\cdot\|_{\Phi(w),B}$. Also, the complementary Young function of $\Phi$ is given by $\bar{\Phi}(t)=e^t-1$ with corresponding Luxemburg norm denoted by $\|\cdot\|_{\exp L(w),B}$. Given a weight $w$ on $\mathbb R^d$, we can also show the weighted version of \eqref{holder}. That is, the following generalized H\"older inequality in the weighted setting
\begin{equation}\label{gholder}
\frac{1}{w(B)}\int_B|f(x)\cdot g(x)|w(x)\,dx\leq C\big\|f\big\|_{L\log L(w),B}\big\|g\big\|_{\exp L(w),B}
\end{equation}
holds for every ball $B$ in $\mathbb R^d$. It is a simple but important observation that for any ball $B$ in $\mathbb R^d$,
\begin{equation*}
\big\|f\big\|_{L(w),B}\leq \big\|f\big\|_{L\log L(w),B}.
\end{equation*}
This is because $t\leq t\cdot(1+\log^+t)$ for all $t>0$. So we have
\begin{equation}\label{small}
\big\|f\big\|_{L(w),B}=\frac{1}{w(B)}\int_B|f(x)|\cdot w(x)\,dx\leq\big\|f\big\|_{L\log L(w),B}.
\end{equation}

In \cite{bong2}, Bongioanni et al.obtained weighted strong $(p,p)$, $1<p<\infty$, and weak $L\log L$ estimates for the commutators of the Riesz transform and its adjoint associated with the Schr\"odinger operator $\mathcal L=-\Delta+V$, where $V$ satisfies some reverse H\"older inequality. Their results can be summarized as follows:

\begin{thm}[\cite{bong2}]\label{cstrong}
Let $1<p<\infty$ and $w\in A^{\rho,\infty}_p$. If $V\in RH_q$ with $q\geq d$, then the commutator operators $[b,\mathcal R]$ and $[b,\mathcal R^{\ast}]$ are all bounded on $L^p(w)$, whenever $b$ belongs to $\mathrm{BMO}_{\rho,\infty}(\mathbb R^d)$.
\end{thm}

\begin{thm}[\cite{bong2}]\label{cweak}
Let $p=1$ and $w\in A^{\rho,\infty}_1$. If $V\in RH_q$ with $q\geq d$ and $b\in\mathrm{BMO}_{\rho,\infty}(\mathbb R^d)$, then for any given $\lambda>0$, there exists a positive constant $C>0$ such that for those functions $f$ such that $\Phi(|f|)\in L^1(w)$,
\begin{equation*}
\begin{split}
&w\big(\big\{x\in\mathbb R^n:|[b,\mathcal R]f(x)|>\lambda\big\}\big)
\leq C\int_{\mathbb R^d}\Phi\left(\frac{|f(x)|}{\lambda}\right)\cdot w(x)\,dx,\\
&w\big(\big\{x\in\mathbb R^n:|[b,\mathcal R^{\ast}]f(x)|>\lambda\big\}\big)
\leq C\int_{\mathbb R^d}\Phi\left(\frac{|f(x)|}{\lambda}\right)\cdot w(x)\,dx,
\end{split}
\end{equation*}
where $\Phi(t)=t\cdot(1+\log^+t)$ and $\log^+t:=\max\{\log t,0\};$ that is,
\begin{equation*}
\log^+t=
\begin{cases}
\log t,\  &\mbox{as}~~t>1;\\
0,\  &\mbox{otherwise}.\\
\end{cases}
\end{equation*}
\end{thm}

In this paper, firstly, we will define some kinds of weighted Morrey spaces related to certain nonnegative potentials.
Secondly, we prove that the Riesz transform $\mathcal R$ and its adjoint $\mathcal R^{\ast}$ are both bounded operators on these new spaces. Finally, we also obtain the weighted boundedness for the commutators $[b,\mathcal R]$ and $[b,\mathcal R^{\ast}]$ defined in \eqref{briesz}.

Throughout this paper $C$ denotes a positive constant not necessarily the same at each occurrence, and a subscript is added when we wish to make
clear its dependence on the parameter in the subscript. We also use $a\approx b$ to denote the equivalence of $a$ and $b$; that is, there exist two positive constants $C_1$, $C_2$ independent of $a,b$ such that $C_1a\leq b\leq C_2a$.

\section{our main results}
In this section, we introduce some types of weighted Morrey spaces related to the potential $V$ and then give our main results.
\begin{defin}
Let $1\leq p<\infty$, $0\leq\kappa<1$ and $w$ be a weight. For given $0<\theta<\infty$, the weighted Morrey space $L^{p,\kappa}_{\rho,\theta}(w)$ is defined to be the set of all $L^p$ locally integrable functions $f$ on $\mathbb R^d$ for which
\begin{equation}\label{morrey1}
\bigg(\frac{1}{w(B)^{\kappa}}\int_B\big|f(x)\big|^pw(x)\,dx\bigg)^{1/p}
\leq C\cdot\left(1+\frac{r}{\rho(x_0)}\right)^{\theta}
\end{equation}
for every ball $B=B(x_0,r)$ in $\mathbb R^d$. A norm for $f\in L^{p,\kappa}_{\rho,\theta}(w)$, denoted by $\|f\|_{L^{p,\kappa}_{\rho,\theta}(w)}$, is given by the infimum of the constants in \eqref{morrey1}, or equivalently,
\begin{equation*}
\|f\|_{L^{p,\kappa}_{\rho,\theta}(w)}:=\sup_B\left(1+\frac{r}{\rho(x_0)}\right)^{-\theta}\bigg(\frac{1}{w(B)^{\kappa}}\int_B\big|f(x)\big|^pw(x)\,dx\bigg)^{1/p}
<\infty,
\end{equation*}
where the supremum is taken over all balls $B$ in $\mathbb R^d$, $x_0$ and $r$ denote the center and radius of $B$ respectively. Define
\begin{equation*}
L^{p,\kappa}_{\rho,\infty}(w):=\bigcup_{\theta>0}L^{p,\kappa}_{\rho,\theta}(w).
\end{equation*}
\end{defin}
Note that this definition does not coincide with the one given in \cite{pan} (see also \cite{tang2} for the unweighted case), but in view of the space $\mathrm{BMO}_{\rho,\infty}(\mathbb R^d)$ defined above it is more natural in our setting. Obviously, if we take $\theta=0$ or $V\equiv0$, then this new space is just the weighted Morrey space $L^{p,\kappa}(w)$, which was first defined by Komori and Shirai in \cite{komori} (see also \cite{wang1}).
\begin{defin}
Let $p=1$, $0\leq\kappa<1$ and $w$ be a weight. For given $0<\theta<\infty$, the weighted weak Morrey space $WL^{1,\kappa}_{\rho,\theta}(w)$ is defined to be the set of all measurable functions $f$ on $\mathbb R^d$ for which
\begin{equation*}
\frac{1}{w(B)^{\kappa}}\sup_{\lambda>0}\lambda\cdot w\big(\big\{x\in B:|f(x)|>\lambda\big\}\big)
\leq C\cdot\left(1+\frac{r}{\rho(x_0)}\right)^{\theta}
\end{equation*}
for every ball $B=B(x_0,r)$ in $\mathbb R^d$, or equivalently,
\begin{equation*}
\|f\|_{WL^{1,\kappa}_{\rho,\theta}(w)}:=\sup_B\left(1+\frac{r}{\rho(x_0)}\right)^{-\theta}\frac{1}{w(B)^{\kappa}}\sup_{\lambda>0}\lambda\cdot w\big(\big\{x\in B:|f(x)|>\lambda\big\}\big)
<\infty.
\end{equation*}
Correspondingly, we define
\begin{equation*}
WL^{1,\kappa}_{\rho,\infty}(w):=\bigcup_{\theta>0}WL^{1,\kappa}_{\rho,\theta}(w).
\end{equation*}
\end{defin}
Clearly, if we take $\theta=0$ or $V\equiv0$, then this space is just the weighted weak Morrey space $WL^{1,\kappa}(w)$ (see \cite{wang2}). According to the above definitions, one has
\begin{equation*}
\begin{cases}
L^{p,\kappa}(w)\subset L^{p,\kappa}_{\rho,\theta_1}(w)\subset L^{p,\kappa}_{\rho,\theta_2}(w);&\\
WL^{1,\kappa}(w)\subset WL^{1,\kappa}_{\rho,\theta_1}(w)\subset WL^{1,\kappa}_{\rho,\theta_2}(w).&
\end{cases}
\end{equation*}
for $0<\theta_1<\theta_2<\infty$. Hence $L^{p,\kappa}(w)\subset L^{p,\kappa}_{\rho,\infty}(w)$ for $(p,\kappa)\in[1,\infty)\times[0,1)$ and $WL^{1,\kappa}(w)\subset WL^{1,\kappa}_{\rho,\infty}(w)$ for $0\leq\kappa<1$.

The space $L^{p,\kappa}_{\rho,\theta}(w)$ (or $WL^{1,\kappa}_{\rho,\theta}(w)$) could be viewed as an extension of weighted (or weak) Lebesgue space (when $\kappa=\theta=0$). Naturally, one may ask the question whether the above conclusions (i.e., Theorems \ref{strong} and \ref{weak} as well as Theorems \ref{cstrong} and \ref{cweak}) still hold if replacing the weighted Lebesgue spaces by the weighted Morrey spaces. In this work, we give a positive answer to this question. Our main results in this work are presented as follows:
\begin{thm}\label{mainthm:1}
Let $1<p<\infty$, $0<\kappa<1$ and $w\in A^{\rho,\infty}_p$. If $V\in RH_q$ with $q\geq d$, then the operators $\mathcal R$ and $\mathcal R^{\ast}$ are both bounded on $L^{p,\kappa}_{\rho,\infty}(w)$.
\end{thm}

\begin{thm}\label{mainthm:2}
Let $p=1$, $0<\kappa<1$ and $w\in A^{\rho,\infty}_1$. If $V\in RH_q$ with $q\geq d$, then the operators $\mathcal R$ and $\mathcal R^{\ast}$ are both bounded from $L^{1,\kappa}_{\rho,\infty}(w)$ into $WL^{1,\kappa}_{\rho,\infty}(w)$.
\end{thm}

\begin{thm}\label{mainthm:3}
Let $1<p<\infty$, $0<\kappa<1$ and $w\in A^{\rho,\infty}_p$. If $V\in RH_q$ with $q\geq d$, then the commutator operators $[b,\mathcal R]$ and $[b,\mathcal R^{\ast}]$ are both bounded on $L^{p,\kappa}_{\rho,\infty}(w)$, whenever $b\in\mathrm{BMO}_{\rho,\infty}(\mathbb R^d)$.
\end{thm}

To deal with the commutators in the endpoint case, we need to consider a new kind of weighted Morrey spaces of $L\log L$ type.
\begin{defin}
Let $p=1$, $0\leq\kappa<1$ and $w$ be a weight. For given $0<\theta<\infty$, the weighted Morrey space $(L\log L)^{1,\kappa}_{\rho,\theta}(w)$ is defined to be the set of all locally integrable functions $f$ on $\mathbb R^d$ for which
\begin{equation*}
w(B)^{1-\kappa}\big\|f\big\|_{L\log L(w),B}
\leq C\cdot\left(1+\frac{r}{\rho(x_0)}\right)^{\theta}
\end{equation*}
for every ball $B=B(x_0,r)$ in $\mathbb R^d$, or
\begin{equation*}
\|f\|_{(L\log L)^{1,\kappa}_{\rho,\theta}(w)}:=\sup_B\left(1+\frac{r}{\rho(x_0)}\right)^{-\theta}w(B)^{1-\kappa}\big\|f\big\|_{L\log L(w),B}
<\infty.
\end{equation*}
\end{defin}

Concerning the continuity properties of $[b,\mathcal R]$ and $[b,\mathcal R^{\ast}]$ in the weighted Morrey spaces of $L\log L$ type, we have
\begin{thm}\label{mainthm:4}
Let $p=1$, $0<\kappa<1$ and $w\in A^{\rho,\infty}_1$. If $V\in RH_q$ with $q\geq d$ and $b\in\mathrm{BMO}_{\rho,\infty}(\mathbb R^d)$, then for any given $\lambda>0$ and any given ball $B=B(x_0,r)$ of $\mathbb R^d$, there exist some constants $C>0$ and $\vartheta>0$ such that the following inequalities
\begin{equation*}
\begin{split}
&\frac{1}{w(B)^{\kappa}}\cdot w\big(\big\{x\in B:|[b,\mathcal R]f(x)|>\lambda\big\}\big)
\leq C\left(1+\frac{r}{\rho(x_0)}\right)^{\vartheta}
\bigg\|\Phi\bigg(\frac{|f|}{\lambda}\bigg)\bigg\|_{(L\log L)^{1,\kappa}_{\rho,\theta}(w)},\\
&\frac{1}{w(B)^{\kappa}}\cdot w\big(\big\{x\in B:|[b,\mathcal R^{\ast}]f(x)|>\lambda\big\}\big)
\leq C\left(1+\frac{r}{\rho(x_0)}\right)^{\vartheta}
\bigg\|\Phi\bigg(\frac{|f|}{\lambda}\bigg)\bigg\|_{(L\log L)^{1,\kappa}_{\rho,\theta}(w)},
\end{split}
\end{equation*}
hold for those functions $f$ such that $\Phi(|f|)\in(L\log L)^{1,\kappa}_{\rho,\theta}(w)$ with some fixed $\theta>0$, where $\Phi(t)=t\cdot(1+\log^+t)$.
\end{thm}
If we denote
\begin{equation*}
(L\log L)^{1,\kappa}_{\rho,\infty}(w):=\bigcup_{\theta>0} (L\log L)^{1,\kappa}_{\rho,\theta}(w),
\end{equation*}
then Theorem \ref{mainthm:4} now tells us that the commutators $[b,\mathcal R]$ and $[b,\mathcal R^{\ast}]$ are both bounded from $(L\log L)^{1,\kappa}_{\rho,\infty}(w)$ into $WL^{1,\kappa}_{\rho,\infty}(w)$, when $b$ is in $\mathrm{BMO}_{\rho,\infty}(\mathbb R^d)$.

\section{Proofs of Theorems $\ref{mainthm:1}$ and $\ref{mainthm:2}$}

In this section, we will prove the conclusions of Theorems \ref{mainthm:1} and \ref{mainthm:2}.

\begin{proof}[Proof of Theorem $\ref{mainthm:1}$]
We denote by $\mathcal T$ either $\mathcal R$ or $\mathcal R^{\ast}$. By definition, we only have to show that for any given ball $B=B(x_0,r)$ of $\mathbb R^d$, there is some $\vartheta>0$ such that
\begin{equation}\label{Main1}
\bigg(\frac{1}{w(B)^{\kappa}}\int_B\big|\mathcal T f(x)\big|^pw(x)\,dx\bigg)^{1/p}\leq C\cdot\left(1+\frac{r}{\rho(x_0)}\right)^{\vartheta}
\end{equation}
holds for any $f\in L^{p,\kappa}_{\rho,\infty}(w)$ with $1<p<\infty$ and $0<\kappa<1$. Suppose that $f\in L^{p,\kappa}_{\rho,\theta}(w)$ for some $\theta>0$ and $w\in A^{\rho,\theta'}_p$ for some $\theta'>0$. We decompose $f$ as
\begin{equation*}
\begin{cases}
f=f_1+f_2\in L^{p,\kappa}_{\rho,\theta}(w);\  &\\
f_1=f\cdot\chi_{2B};\  &\\
f_2=f\cdot\chi_{(2B)^c},
\end{cases}
\end{equation*}
where $2B$ is the ball centered at $x_0$ and radius $2r>0$, and $\chi_{2B}$ is the characteristic function of $2B$. Then by the linearity of $\mathcal T$, we write
\begin{equation*}
\begin{split}
\bigg(\frac{1}{w(B)^{\kappa}}\int_B\big|\mathcal Tf(x)\big|^pw(x)\,dx\bigg)^{1/p}
&\leq\bigg(\frac{1}{w(B)^{\kappa}}\int_B\big|\mathcal Tf_1(x)\big|^pw(x)\,dx\bigg)^{1/p}\\
&+\bigg(\frac{1}{w(B)^{\kappa}}\int_B\big|\mathcal Tf_2(x)\big|^pw(x)\,dx\bigg)^{1/p}\\
&:=I_1+I_2.
\end{split}
\end{equation*}
We now analyze each term separately. By Theorem \ref{strong}, we get
\begin{equation*}
\begin{split}
I_1&=\bigg(\frac{1}{w(B)^{\kappa}}\int_B\big|\mathcal Tf_1(x)\big|^pw(x)\,dx\bigg)^{1/p}\\
&\leq C\cdot\frac{1}{w(B)^{\kappa/p}}\bigg(\int_{\mathbb R^d}\big|f_1(x)\big|^pw(x)\,dx\bigg)^{1/p}\\
&=C\cdot\frac{1}{w(B)^{\kappa/p}}\bigg(\int_{2B}\big|f(x)\big|^pw(x)\,dx\bigg)^{1/p}\\
&\leq C\|f\|_{L^{p,\kappa}_{\rho,\theta}(w)}\cdot\frac{w(2B)^{\kappa/p}}{w(B)^{\kappa/p}}\cdot\left(1+\frac{2r}{\rho(x_0)}\right)^{\theta}.
\end{split}
\end{equation*}
Since $w\in A^{\rho,\theta'}_p$ with $1<p<\infty$ and $0<\theta'<\infty$, then we know that the following inequality
\begin{equation}\label{doubling1}
w\big(2B(x_0,r)\big)\leq C\cdot\left(1+\frac{2r}{\rho(x_0)}\right)^{p\theta'}w\big(B(x_0,r)\big)
\end{equation}
is valid. In fact, for $1<p<\infty$, by H\"older's inequality and the definition of $A^{\rho,\theta'}_p$, we have
\begin{equation*}
\begin{split}
&\frac{1}{|2B|}\int_{2B}\big|\hbar(x)\big|\,dx=\frac{1}{|2B|}\int_{2B}\big|\hbar(x)\big|w(x)^{1/p}w(x)^{-1/p}\,dx\\
&\leq\frac{1}{|2B|}\bigg(\int_{2B}\big|\hbar(x)\big|^pw(x)\,dx\bigg)^{1/p}
\bigg(\int_{2B}w(x)^{-{p'}/p}\,dx\bigg)^{1/{p'}}\\
&\leq \frac{C_w}{w(2B)^{1/p}}\bigg(\int_{2B}\big|\hbar(x)\big|^pw(x)\,dx\bigg)^{1/p}
\left(1+\frac{2r}{\rho(x_0)}\right)^{\theta'}.
\end{split}
\end{equation*}
If we take $\hbar(x)=\chi_B(x)$, then the above expression becomes
\begin{equation}\label{W1}
\frac{|B|}{|2B|}\leq C_w\cdot\frac{w(B)^{1/p}}{w(2B)^{1/p}}\left(1+\frac{2r}{\rho(x_0)}\right)^{\theta'},
\end{equation}
which in turn implies \eqref{doubling1}. Therefore,
\begin{equation*}
\begin{split}
I_1
&\leq C\|f\|_{L^{p,\kappa}_{\rho,\theta}(w)}
\cdot\left(1+\frac{2r}{\rho(x_0)}\right)^{(p\theta')\cdot(\kappa/p)}\cdot\left(1+\frac{2r}{\rho(x_0)}\right)^{\theta}\\
&=C\|f\|_{L^{p,\kappa}_{\rho,\theta}(w)}
\cdot\left(1+\frac{2r}{\rho(x_0)}\right)^{\vartheta'}\leq C\cdot\left(1+\frac{r}{\rho(x_0)}\right)^{\vartheta'},
\end{split}
\end{equation*}
where $\vartheta':=\kappa\cdot\theta'+\theta$. For the other term $I_2$, notice that for any $x\in B$ and $y\in (2B)^c$, one has $|x-y|\approx|x_0-y|$. It then follows from Lemma \ref{kernel} that for any $x\in B(x_0,r)$ and any positive integer $N$,
\begin{equation}\label{T}
\begin{split}
\big|\mathcal Tf_2(x)\big|&\leq\int_{(2B)^c}|\mathcal K(x,y)|\cdot|f(y)|\,dy~(\hbox{or}~~ \mathcal K^{\ast}(x,y))\\
&\leq C_N\int_{(2B)^c}\bigg(1+\frac{|x-y|}{\rho(x)}\bigg)^{-N}\frac{1}{|x-y|^d}\cdot|f(y)|\,dy\\
&\leq C_{N,d}\int_{(2B)^c}\bigg(1+\frac{|x_0-y|}{\rho(x)}\bigg)^{-N}\frac{1}{|x_0-y|^d}\cdot|f(y)|\,dy\\
&=C_{N,d}\sum_{k=1}^\infty\int_{2^{k+1}B\backslash 2^k B}\bigg(1+\frac{|x_0-y|}{\rho(x)}\bigg)^{-N}\frac{1}{|x_0-y|^d}\cdot|f(y)|\,dy\\
&\leq C\sum_{k=1}^\infty\frac{1}{|2^{k+1}B|}\int_{2^{k+1}B\backslash 2^k B}\bigg(1+\frac{2^kr}{\rho(x)}\bigg)^{-N}|f(y)|\,dy.
\end{split}
\end{equation}
In view of \eqref{com2} in Lemma \ref{N0}, we further obtain
\begin{align}\label{Tf2}
\big|\mathcal Tf_2(x)\big|
&\leq C\sum_{k=1}^\infty\frac{1}{|2^{k+1}B|}\int_{2^{k+1}B}\left(1+\frac{r}{\rho(x_0)}\right)^{N\cdot\frac{N_0}{N_0+1}}
\left(1+\frac{2^kr}{\rho(x_0)}\right)^{-N}|f(y)|\,dy\notag\\
&\leq C\sum_{k=1}^\infty\frac{1}{|2^{k+1}B|}\int_{2^{k+1}B}\left(1+\frac{r}{\rho(x_0)}\right)^{N\cdot\frac{N_0}{N_0+1}}
\left(1+\frac{2^{k+1}r}{\rho(x_0)}\right)^{-N}|f(y)|\,dy.
\end{align}
Moreover, by using H\"older's inequality and $A^{\rho,\theta'}_p$ condition on $w$, we get
\begin{equation*}
\begin{split}
&\frac{1}{|2^{k+1}B|}\int_{2^{k+1}B}\big|f(y)\big|\,dy\\
&\leq\frac{1}{|2^{k+1}B|}\bigg(\int_{2^{k+1}B}\big|f(y)\big|^pw(y)\,dy\bigg)^{1/p}
\bigg(\int_{2^{k+1}B}w(y)^{-{p'}/p}\,dy\bigg)^{1/{p'}}\\
&\leq C\|f\|_{L^{p,\kappa}_{\rho,\theta}(w)}\cdot\frac{w(2^{k+1}B)^{{\kappa}/p}}{w(2^{k+1}B)^{1/p}}
\left(1+\frac{2^{k+1}r}{\rho(x_0)}\right)^{\theta}\left(1+\frac{2^{k+1}r}{\rho(x_0)}\right)^{\theta'}.
\end{split}
\end{equation*}
Hence,
\begin{equation*}
\begin{split}
I_2&\leq C\|f\|_{L^{p,\kappa}_{\rho,\theta}(w)}\cdot\frac{w(B)^{1/p}}{w(B)^{{\kappa}/p}}
\sum_{k=1}^\infty\frac{w(2^{k+1}B)^{{\kappa}/p}}{w(2^{k+1}B)^{1/p}}
\left(1+\frac{r}{\rho(x_0)}\right)^{N\cdot\frac{N_0}{N_0+1}}\left(1+\frac{2^{k+1}r}{\rho(x_0)}\right)^{-N+\theta+\theta'}\\
&=C\|f\|_{L^{p,\kappa}_{\rho,\theta}(w)}\left(1+\frac{r}{\rho(x_0)}\right)^{N\cdot\frac{N_0}{N_0+1}}
\sum_{k=1}^\infty\frac{w(B)^{{(1-\kappa)}/p}}{w(2^{k+1}B)^{{(1-\kappa)}/p}}
\left(1+\frac{2^{k+1}r}{\rho(x_0)}\right)^{-N+\theta+\theta'}.
\end{split}
\end{equation*}
Recall that $w\in A^{\rho,\theta'}_p$ with $0<\theta'<\infty$ and $1<p<\infty$, then there exist two positive numbers $\delta,\eta>0$ such that \eqref{compare} holds. This allows us to obtain
\begin{equation*}
\begin{split}
I_2&\leq C\|f\|_{L^{p,\kappa}_{\rho,\theta}(w)}
\left(1+\frac{r}{\rho(x_0)}\right)^{N\cdot\frac{N_0}{N_0+1}}\sum_{k=1}^\infty\left(\frac{|B|}{|2^{k+1}B|}\right)^{\delta{(1-\kappa)}/p}\\
&\quad\times\left(1+\frac{2^{k+1}r}{\rho(x_0)}\right)^{\eta{(1-\kappa)}/p}\left(1+\frac{2^{k+1}r}{\rho(x_0)}\right)^{-N+\theta+\theta'}\\
&=C\|f\|_{L^{p,\kappa}_{\rho,\theta}(w)}
\left(1+\frac{r}{\rho(x_0)}\right)^{N\cdot\frac{N_0}{N_0+1}}\sum_{k=1}^\infty\left(\frac{|B|}{|2^{k+1}B|}\right)^{\delta{(1-\kappa)}/p}
\left(1+\frac{2^{k+1}r}{\rho(x_0)}\right)^{-N+\theta+\theta'+\eta{(1-\kappa)}/p}.
\end{split}
\end{equation*}
Thus, by choosing $N$ large enough so that $N>\theta+\theta'+\eta{(1-\kappa)}/p$, we then have
\begin{equation*}
\begin{split}
I_2&\leq C\|f\|_{L^{p,\kappa}_{\rho,\theta}(w)}
\left(1+\frac{r}{\rho(x_0)}\right)^{N\cdot\frac{N_0}{N_0+1}}\sum_{k=1}^\infty\left(\frac{|B|}{|2^{k+1}B|}\right)^{\delta{(1-\kappa)}/p}\\
&\leq C\left(1+\frac{r}{\rho(x_0)}\right)^{N\cdot\frac{N_0}{N_0+1}}.
\end{split}
\end{equation*}
Summing up the above estimates for $I_1$ and $I_2$ and letting $\vartheta=\max\big\{\vartheta',N\cdot\frac{N_0}{N_0+1}\big\}$, we obtain our desired inequality \eqref{Main1}. This completes the proof of Theorem \ref{mainthm:1}.
\end{proof}

\begin{proof}[Proof of Theorem $\ref{mainthm:2}$]
We denote by $\mathcal T$ either $\mathcal R$ or $\mathcal R^{\ast}$. To prove Theorem \ref{mainthm:2}, by definition, it suffices to prove that for any given ball $B=B(x_0,r)$ of $\mathbb R^d$, there is some $\vartheta>0$ such that
\begin{equation}\label{Main2}
\frac{1}{w(B)^{\kappa}}\sup_{\lambda>0}\lambda\cdot w\big(\big\{x\in B:|\mathcal T f(x)|>\lambda\big\}\big)\leq C\cdot\left(1+\frac{r}{\rho(x_0)}\right)^{\vartheta}
\end{equation}
holds for any $f\in L^{1,\kappa}_{\rho,\infty}(w)$ with $0<\kappa<1$. Now suppose that $f\in L^{1,\kappa}_{\rho,\theta}(w)$ for some $\theta>0$ and $w\in A^{\rho,\theta'}_1$ for some $\theta'>0$. We decompose $f$ as
\begin{equation*}
\begin{cases}
f=f_1+f_2\in L^{1,\kappa}_{\rho,\theta}(w);\  &\\
f_1=f\cdot\chi_{2B};\  &\\
f_2=f\cdot\chi_{(2B)^c}.
\end{cases}
\end{equation*}
Then for any given $\lambda>0$, by the linearity of $\mathcal T$, we can write
\begin{equation*}
\begin{split}
\frac{1}{w(B)^{\kappa}}\lambda\cdot w\big(\big\{x\in B:|\mathcal T f(x)|>\lambda\big\}\big)
&\leq\frac{1}{w(B)^{\kappa}}\lambda\cdot w\big(\big\{x\in B:|\mathcal T f_1(x)|>\lambda/2\big\}\big)\\
&+\frac{1}{w(B)^{\kappa}}\lambda\cdot w\big(\big\{x\in B:|\mathcal T f_2(x)|>\lambda/2\big\}\big)\\
&:=I'_1+I'_2.
\end{split}
\end{equation*}
We first give the estimate for the term $I'_1$. By Theorem \ref{weak}, we get
\begin{equation*}
\begin{split}
I'_1&=\frac{1}{w(B)^{\kappa}}\lambda\cdot w\big(\big\{x\in B:|\mathcal T f_1(x)|>\lambda/2\big\}\big)\\
&\leq C\cdot\frac{1}{w(B)^{\kappa}}\bigg(\int_{\mathbb R^d}\big|f_1(x)\big|w(x)\,dx\bigg)\\
&=C\cdot\frac{1}{w(B)^{\kappa}}\bigg(\int_{2B}\big|f(x)\big|w(x)\,dx\bigg)\\
&\leq C\|f\|_{L^{1,\kappa}_{\rho,\theta}(w)}\cdot\frac{w(2B)^{\kappa}}{w(B)^{\kappa}}\cdot\left(1+\frac{2r}{\rho(x_0)}\right)^{\theta}.
\end{split}
\end{equation*}
Since $w\in A^{\rho,\theta'}_1$ with $0<\theta'<\infty$, similar to the proof of \eqref{doubling1}, we can also show the following estimate as well.
\begin{equation}\label{doubling2}
w\big(2B(x_0,r)\big)\leq C\cdot\left(1+\frac{2r}{\rho(x_0)}\right)^{\theta'}w\big(B(x_0,r)\big).
\end{equation}
In fact, by the definition of $A^{\rho,\theta'}_1$, we can deduce that
\begin{equation*}
\begin{split}
\frac{1}{|2B|}\int_{2B}\big|\hbar(x)\big|\,dx
&\leq \frac{C_w}{w(2B)}\cdot\underset{x\in 2B}{\mbox{ess\,inf}}\;w(x)
\bigg(\int_{2B}\big|\hbar(x)\big|\,dx\bigg)\left(1+\frac{2r}{\rho(x_0)}\right)^{\theta'}\\
&\leq \frac{C_w}{w(2B)}\bigg(\int_{2B}\big|\hbar(x)\big|w(x)\,dx\bigg)\left(1+\frac{2r}{\rho(x_0)}\right)^{\theta'}.
\end{split}
\end{equation*}
If we choose $\hbar(x)=\chi_B(x)$, then the above expression becomes
\begin{equation}\label{W2}
\frac{|B|}{|2B|}\leq C_w\cdot\frac{w(B)}{w(2B)}\left(1+\frac{2r}{\rho(x_0)}\right)^{\theta'},
\end{equation}
which in turn implies \eqref{doubling2}. Therefore,
\begin{equation*}
\begin{split}
I'_1
&\leq C\cdot\left(1+\frac{2r}{\rho(x_0)}\right)^{\theta'\cdot\kappa}\cdot\left(1+\frac{2r}{\rho(x_0)}\right)^{\theta}\\
&=C\cdot\left(1+\frac{2r}{\rho(x_0)}\right)^{\vartheta'}\leq C\cdot\left(1+\frac{r}{\rho(x_0)}\right)^{\vartheta'},
\end{split}
\end{equation*}
where $\vartheta':=\theta'\cdot\kappa+\theta$. As for the other term $I'_2$, by using the pointwise inequality \eqref{Tf2} and Chebyshev's inequality, we deduce that
\begin{equation*}
\begin{split}
I'_2&=\frac{1}{w(B)^{\kappa}}\lambda\cdot w\big(\big\{x\in B:|\mathcal T f_2(x)|>\lambda/2\big\}\big)\\
&\leq\frac{2}{w(B)^{\kappa}}\bigg(\int_{B}\big|\mathcal Tf_2(x)\big|w(x)\,dx\bigg)\\
&\leq C\cdot\frac{w(B)}{w(B)^{\kappa}}
\sum_{k=1}^\infty\frac{1}{|2^{k+1}B|}\int_{2^{k+1}B}\left(1+\frac{r}{\rho(x_0)}\right)^{N\cdot\frac{N_0}{N_0+1}}
\left(1+\frac{2^{k+1}r}{\rho(x_0)}\right)^{-N}|f(y)|\,dy.
\end{split}
\end{equation*}
Moreover, by the $A^{\rho,\theta'}_1$ condition on $w$, we compute
\begin{equation*}
\begin{split}
&\frac{1}{|2^{k+1}B|}\int_{2^{k+1}B}\big|f(y)\big|\,dy\\
&\leq \frac{C_w}{w(2^{k+1}B)}\cdot\underset{y\in 2^{k+1}B}{\mbox{ess\,inf}}\;w(y)
\bigg(\int_{2^{k+1}B}\big|f(y)\big|\,dy\bigg)\left(1+\frac{2^{k+1}r}{\rho(x_0)}\right)^{\theta'}\\
&\leq \frac{C_w}{w(2^{k+1}B)}\bigg(\int_{2^{k+1}B}\big|f(y)\big|w(y)\,dy\bigg)\left(1+\frac{2^{k+1}r}{\rho(x_0)}\right)^{\theta'}\\
&\leq C\|f\|_{L^{1,\kappa}_{\rho,\theta}(w)}\cdot\frac{w(2^{k+1}B)^{\kappa}}{w(2^{k+1}B)}
\left(1+\frac{2^{k+1}r}{\rho(x_0)}\right)^{\theta}\left(1+\frac{2^{k+1}r}{\rho(x_0)}\right)^{\theta'}.
\end{split}
\end{equation*}
Consequently,
\begin{equation*}
\begin{split}
I'_2&\leq C\|f\|_{L^{1,\kappa}_{\rho,\theta}(w)}
\cdot\frac{w(B)}{w(B)^{\kappa}}\sum_{k=1}^\infty\frac{w(2^{k+1}B)^{\kappa}}{w(2^{k+1}B)}
\left(1+\frac{r}{\rho(x_0)}\right)^{N\cdot\frac{N_0}{N_0+1}}\left(1+\frac{2^{k+1}r}{\rho(x_0)}\right)^{-N+\theta+\theta'}\\
&=C\|f\|_{L^{1,\kappa}_{\rho,\theta}(w)}
\left(1+\frac{r}{\rho(x_0)}\right)^{N\cdot\frac{N_0}{N_0+1}}\sum_{k=1}^\infty\frac{w(B)^{{(1-\kappa)}}}{w(2^{k+1}B)^{{(1-\kappa)}}}
\left(1+\frac{2^{k+1}r}{\rho(x_0)}\right)^{-N+\theta+\theta'}.
\end{split}
\end{equation*}
Recall that $w\in A^{\rho,\theta'}_1$ with $0<\theta'<\infty$, then there exist two positive numbers $\delta',\eta'>0$ such that \eqref{compare} holds. Therefore,
\begin{equation*}
\begin{split}
I'_2&\leq C\|f\|_{L^{1,\kappa}_{\rho,\theta}(w)}
\left(1+\frac{r}{\rho(x_0)}\right)^{N\cdot\frac{N_0}{N_0+1}}\sum_{k=1}^\infty\left(\frac{|B|}{|2^{k+1}B|}\right)^{\delta'{(1-\kappa)}}\\
&\quad\times\left(1+\frac{2^{k+1}r}{\rho(x_0)}\right)^{\eta'{(1-\kappa)}}\left(1+\frac{2^{k+1}r}{\rho(x_0)}\right)^{-N+\theta+\theta'}\\
&=C\|f\|_{L^{1,\kappa}_{\rho,\theta}(w)}
\left(1+\frac{r}{\rho(x_0)}\right)^{N\cdot\frac{N_0}{N_0+1}}\\
&\quad\times\sum_{k=1}^\infty\left(\frac{|B|}{|2^{k+1}B|}\right)^{\delta'{(1-\kappa)}}
\left(1+\frac{2^{k+1}r}{\rho(x_0)}\right)^{-N+\theta+\theta'+\eta'{(1-\kappa)}}.
\end{split}
\end{equation*}
By selecting $N$ large enough so that $N>\theta+\theta'+\eta'{(1-\kappa)}$, we thus have
\begin{equation*}
\begin{split}
I'_2&\leq C\left(1+\frac{r}{\rho(x_0)}\right)^{N\cdot\frac{N_0}{N_0+1}}\sum_{k=1}^\infty\left(\frac{|B|}{|2^{k+1}B|}\right)^{\delta'{(1-\kappa)}}\\
&\leq C\left(1+\frac{r}{\rho(x_0)}\right)^{N\cdot\frac{N_0}{N_0+1}}.
\end{split}
\end{equation*}
Let $\vartheta=\max\big\{\vartheta',N\cdot\frac{N_0}{N_0+1}\big\}$. Here $N$ is an appropriate constant. Summing up the above estimates for $I'_1$ and $I'_2$, and then taking the supremum over all $\lambda>0$, we obtain our desired inequality \eqref{Main2}. This finishes the proof of Theorem \ref{mainthm:2}.
\end{proof}

\section{Proofs of Theorems $\ref{mainthm:3}$ and $\ref{mainthm:4}$}

For the results involving commutators, we need the following properties of $\mathrm{BMO}_{\rho,\infty}(\mathbb R^d)$ functions, which are extensions of well-known properties of $\mathrm{BMO}(\mathbb R^d)$ functions.
\begin{lem}\label{BMO1}
If $b\in \mathrm{BMO}_{\rho,\infty}(\mathbb R^d)$ and $w\in A^{\rho,\infty}_p$ with $1\leq p<\infty$, then there exist positive constants $C>0$ and $\mu>0$ such that for every ball $B=B(x_0,r)$ in $\mathbb R^d$, we have
\begin{equation}\label{wang3}
\bigg(\int_B\big|b(x)-b_B\big|^pw(x)\,dx\bigg)^{1/p}\leq C\cdot w(B)^{1/p}\left(1+\frac{r}{\rho(x_0)}\right)^{\mu},
\end{equation}
where $b_B=\frac{1}{|B|}\int_B b(y)\,dy$.
\end{lem}
\begin{proof}
We may assume that $b\in \mathrm{BMO}_{\rho,\theta}(\mathbb R^d)$ with $0<\theta<\infty$. According to Proposition \ref{wangbmo}, we can deduce that
\begin{equation*}
\begin{split}
&\bigg(\int_B\big|b(x)-b_B\big|^pw(x)\,dx\bigg)^{1/p}\\
&=\bigg(\int_0^{\infty}p\lambda^{p-1}w\big(\big\{x\in B:|b(x)-b_{B}|>\lambda\big\}\big)d\lambda\bigg)^{1/p}\\
&\leq C_1^{1/p}\cdot w(B)^{1/p}\bigg\{\int_0^{\infty}p\lambda^{p-1}\exp\bigg[-\bigg(1+\frac{r}{\rho(x_0)}\bigg)^{-\theta^{\ast}}\frac{C_2 \lambda}{\|b\|_{\mathrm{BMO}_{\rho,\theta}}}\bigg]
\left(1+\frac{r}{\rho(x_0)}\right)^{\eta}d\lambda\bigg\}^{1/p}\\
&=C_1^{1/p}\cdot w(B)^{1/p}
\bigg\{\int_0^{\infty}p\lambda^{p-1}\exp\bigg[-\bigg(1+\frac{r}{\rho(x_0)}\bigg)^{-\theta^{\ast}}\frac{C_2 \lambda}{\|b\|_{\mathrm{BMO}_{\rho,\theta}}}\bigg]d\lambda\bigg\}^{1/p}
\left(1+\frac{r}{\rho(x_0)}\right)^{\eta/p}.
\end{split}
\end{equation*}
Making change of variables, then we get
\begin{equation*}
\begin{split}
&\bigg(\int_B\big|b(x)-b_B\big|^pw(x)\,dx\bigg)^{1/p}\\
&\leq C_1^{1/p}\cdot w(B)^{1/p}\bigg(\int_0^{\infty}ps^{p-1}e^{-s}ds\bigg)^{1/p}
\left(\frac{\|b\|_{\mathrm{BMO}_{\rho,\theta}}}{C_2}\right)\left(1+\frac{r}{\rho(x_0)}\right)^{\theta^{\ast}}\left(1+\frac{r}{\rho(x_0)}\right)^{\eta/p}\\
&=\big[C_1p\Gamma(p)\big]^{1/p}\left(\frac{\|b\|_{\mathrm{BMO}_{\rho,\theta}}}{C_2}\right)\cdot w(B)^{1/p}
\left(1+\frac{r}{\rho(x_0)}\right)^{\theta^{\ast}+\eta/p},
\end{split}
\end{equation*}
which yields the desired inequality if we choose $C=\big[C_1p\Gamma(p)\big]^{1/p}\left(\frac{\|b\|_{\mathrm{BMO}_{\rho,\theta}}}{C_2}\right)$ and $\mu=\theta^{\ast}+\eta/p$.
\end{proof}

\begin{lem}\label{BMO3}
If $b\in \mathrm{BMO}_{\rho,\theta}(\mathbb R^d)$ with $0<\theta<\infty$ and $w\in A^{\rho,\infty}_1$, then there exist positive constants $C,\gamma>0$ and $\eta>0$ such that for every ball $B=B(x_0,r)$ in $\mathbb R^d$, we have
\begin{equation}\label{wang2}
\begin{split}
&\bigg(\int_B\bigg\{\exp\bigg[\bigg(1+\frac{r}{\rho(x_0)}\bigg)^{-\theta^{\ast}}\frac{\gamma}{\|b\|_{\mathrm{BMO}_{\rho,\theta}}}
\big|b(x)-b_B\big|\bigg]-1\bigg\}w(x)\,dx\bigg)\\
&\leq C\cdot w(B)\left(1+\frac{r}{\rho(x_0)}\right)^{\eta},
\end{split}
\end{equation}
where $b_B=\frac{1}{|B|}\int_B b(y)\,dy$ and $\theta^{\ast}=(N_0+1)\theta$ and $N_0$ is the constant appearing in Lemma \ref{N0}.
\end{lem}
\begin{proof}
Recall the following identity (see Proposition 1.1.4 in \cite{grafakos})
\begin{equation*}
\bigg(\int_B\big\{\exp\big[f(x)\big]-1\big\}w(x)\,dx\bigg)
=\int_0^{\infty}e^\lambda w\big(\big\{x\in B:|f(x)|>\lambda\big\}\big)d\lambda.
\end{equation*}
Using this identity and Proposition \ref{wangbmo}, we obtain
\begin{equation*}
\begin{split}
&\bigg(\int_B\bigg\{\exp\bigg[\bigg(1+\frac{r}{\rho(x_0)}\bigg)^{-\theta^{\ast}}\frac{\gamma}{\|b\|_{\mathrm{BMO}_{\rho,\theta}}}
\big|b(x)-b_B\big|\bigg]-1\bigg\}w(x)\,dx\bigg)\\
&=\int_0^{\infty}e^\lambda w\big(\big\{x\in B:\big|b(x)-b_B\big|>\lambda^{\ast}\big\}\big)d\lambda\\
&\leq C_1\cdot w(B)\int_0^{\infty}e^\lambda
\exp\bigg[-\bigg(1+\frac{r}{\rho(x_0)}\bigg)^{-\theta^{\ast}}\frac{C_2 \lambda^{\ast}}{\|b\|_{\mathrm{BMO}_{\rho,\theta}}}\bigg]d\lambda
\left(1+\frac{r}{\rho(x_0)}\right)^{\eta}\\
&=C_1\cdot w(B)\int_0^{\infty}e^\lambda\cdot e^{-\frac{C_2\lambda}{\gamma}}d\lambda
\left(1+\frac{r}{\rho(x_0)}\right)^{\eta},
\end{split}
\end{equation*}
where $\lambda^{\ast}$ is given by
\begin{equation*}
\lambda^{\ast}=\frac{\lambda\|b\|_{\mathrm{BMO}_{\rho,\theta}}}{\gamma}\bigg(1+\frac{r}{\rho(x_0)}\bigg)^{\theta^{\ast}}.
\end{equation*}
If we take $\gamma$ small enough so that $0<\gamma<C_2$, then the conclusion follows immediately.
\end{proof}

\begin{lem}\label{BMO2}
If $b\in \mathrm{BMO}_{\rho,\theta}(\mathbb R^d)$ with $0<\theta<\infty$, then for any positive integer $k$, there exists a positive constant $C>0$ such that for every ball $B=B(x_0,r)$ in $\mathbb R^d$, we have
\begin{equation*}
\big|b_{2^{k+1}B}-b_B\big|\leq C\cdot(k+1)\left(1+\frac{2^{k+1}r}{\rho(x_0)}\right)^{\theta}.
\end{equation*}
\end{lem}
\begin{proof}
For any positive integer $k$, we have
\begin{equation*}
\begin{split}
\big|b_{2^{k+1}B}-b_B\big|&\leq \sum_{j=1}^{k+1}\big|b_{2^{j}B}-b_{2^{j-1}B}\big|\\
&=\sum_{j=1}^{k+1}\bigg|\frac{1}{|2^{j-1}B|}\int_{2^{j-1}B}\big[b_{2^{j}B}-b(y)\big]dy\bigg|\\
&\leq\sum_{j=1}^{k+1}\frac{2^d}{|2^{j}B|}\int_{2^{j}B}\big|b(y)-b_{2^{j}B}\big|dy\\
&\leq C_{b,d}\|b\|_{\mathrm{BMO}_{\rho,\theta}}\sum_{j=1}^{k+1}\left(1+\frac{2^jr}{\rho(x_0)}\right)^{\theta}.
\end{split}
\end{equation*}
Since for any $1\leq j\leq k+1$, the following estimate
\begin{equation*}
\left(1+\frac{2^jr}{\rho(x_0)}\right)^{\theta}\leq\left(1+\frac{2^{k+1}r}{\rho(x_0)}\right)^{\theta},
\end{equation*}
holds trivially, and hence
\begin{equation*}
\big|b_{2^{k+1}B}-b_B\big|\leq C\sum_{j=1}^{k+1}\left(1+\frac{2^{k+1}r}{\rho(x_0)}\right)^{\theta}=C\cdot(k+1)\left(1+\frac{2^{k+1}r}{\rho(x_0)}\right)^{\theta}.
\end{equation*}
We obtain the desired result. This completes the proof.
\end{proof}

Now, we are in a position to prove our main results in this section.

\begin{proof}[Proof of Theorem $\ref{mainthm:3}$]
We denote by $[b,\mathcal T]$ either $[b,\mathcal R]$ or $[b,\mathcal R^{\ast}]$. By definition, we only need to show that for any given ball $B=B(x_0,r)$ of $\mathbb R^d$, there is some $\vartheta>0$ such that
\begin{equation}\label{Main3}
\bigg(\frac{1}{w(B)^{\kappa}}\int_B\big|[b,\mathcal T]f(x)\big|^pw(x)\,dx\bigg)^{1/p}\leq C\cdot\left(1+\frac{r}{\rho(x_0)}\right)^{\vartheta}
\end{equation}
holds for any $f\in L^{p,\kappa}_{\rho,\infty}(w)$ with $1<p<\infty$ and $0<\kappa<1$, whenever $b$ belongs to $\mathrm{BMO}_{\rho,\infty}(\mathbb R^d)$. Suppose that $f\in L^{p,\kappa}_{\rho,\theta}(w)$ for some $\theta>0$, $w\in A^{\rho,\theta'}_p$ for some $\theta'>0$ as well as $b\in \mathrm{BMO}_{\rho,\theta''}(\mathbb R^d)$ for some $\theta''>0$. We decompose $f$ as
\begin{equation*}
\begin{cases}
f=f_1+f_2\in L^{p,\kappa}_{\rho,\theta}(w);\  &\\
f_1=f\cdot\chi_{2B};\  &\\
f_2=f\cdot\chi_{(2B)^c}.
\end{cases}
\end{equation*}
Then by the linearity of $[b,\mathcal T]$, we write
\begin{equation*}
\begin{split}
\bigg(\frac{1}{w(B)^{\kappa}}\int_B\big|[b,\mathcal T]f(x)\big|^pw(x)\,dx\bigg)^{1/p}
&\leq\bigg(\frac{1}{w(B)^{\kappa}}\int_B\big|[b,\mathcal T]f_1(x)\big|^pw(x)\,dx\bigg)^{1/p}\\
&+\bigg(\frac{1}{w(B)^{\kappa}}\int_B\big|[b,\mathcal T]f_2(x)\big|^pw(x)\,dx\bigg)^{1/p}\\
&:=J_1+J_2.
\end{split}
\end{equation*}
Now we give the estimates for $J_1$, $J_2$, respectively. According to Theorem \ref{cstrong}, we have
\begin{equation*}
\begin{split}
J_1&\leq C\cdot\frac{1}{w(B)^{\kappa/p}}\bigg(\int_{\mathbb R^d}\big|f_1(x)\big|^pw(x)\,dx\bigg)^{1/p}\\
&=C\cdot\frac{1}{w(B)^{\kappa/p}}\bigg(\int_{2B}\big|f(x)\big|^pw(x)\,dx\bigg)^{1/p}\\
&\leq C\|f\|_{L^{p,\kappa}_{\rho,\theta}(w)}
\cdot\frac{w(2B)^{\kappa/p}}{w(B)^{\kappa/p}}\cdot\left(1+\frac{2r}{\rho(x_0)}\right)^{\theta}.
\end{split}
\end{equation*}
Moreover, in view of the inequality \eqref{doubling1}, we get
\begin{equation*}
\begin{split}
J_1&\leq C\|f\|_{L^{p,\kappa}_{\rho,\theta}(w)}
\cdot\left(1+\frac{2r}{\rho(x_0)}\right)^{(p\theta')\cdot(\kappa/p)}\cdot\left(1+\frac{2r}{\rho(x_0)}\right)^{\theta}\\
&\leq C\cdot\left(1+\frac{r}{\rho(x_0)}\right)^{\vartheta'},
\end{split}
\end{equation*}
where $\vartheta':=\theta'\cdot\kappa+\theta$. On the other hand, by the definition \eqref{briesz}, we can see that for any $x\in B(x_0,r)$,
\begin{equation}\label{twoterm}
\begin{split}
\big|[b,\mathcal T]f_2(x)\big|&\leq\int_{\mathbb R^d}\big|b(x)-b(y)\big|\big|\mathcal K(x,y)f_2(y)\big|\,dy~(\hbox{or}~~ \mathcal K^{\ast}(x,y))\\
&\leq\big|b(x)-b_B\big|\int_{\mathbb R^d}\big|\mathcal K(x,y)f_2(y)\big|\,dy+\int_{\mathbb R^d}\big|b(y)-b_B\big|\big|\mathcal K(x,y)f_2(y)\big|\,dy\\
&:=\xi(x)+\zeta(x).
\end{split}
\end{equation}
So we can divide $J_2$ into two parts:
\begin{equation*}
\begin{split}
J_2&\leq\bigg(\frac{1}{w(B)^{\kappa}}\int_B\big|\xi(x)\big|^pw(x)\,dx\bigg)^{1/p}
+\bigg(\frac{1}{w(B)^{\kappa}}\int_B\big|\zeta(x)\big|^pw(x)\,dx\bigg)^{1/p}\\
&:=J_3+J_4.
\end{split}
\end{equation*}
From the pointwise estimate \eqref{Tf2} and \eqref{wang3} in Lemma \ref{BMO1}, it then follows that
\begin{equation*}
\begin{split}
J_3&\leq C\cdot\frac{1}{w(B)^{\kappa/p}}\bigg(\int_B\big|b(x)-b_B\big|^pw(x)\,dx\bigg)^{1/p}\\
&\quad\times\sum_{k=1}^\infty\frac{1}{|2^{k+1}B|}\int_{2^{k+1}B}\left(1+\frac{r}{\rho(x_0)}\right)^{N\cdot\frac{N_0}{N_0+1}}
\left(1+\frac{2^{k+1}r}{\rho(x_0)}\right)^{-N}|f(y)|\,dy\\
&\leq C_b\cdot\frac{w(B)^{1/p}}{w(B)^{\kappa/p}}\left(1+\frac{r}{\rho(x_0)}\right)^{\mu}\\
&\quad\times\sum_{k=1}^\infty\frac{1}{|2^{k+1}B|}\int_{2^{k+1}B}\left(1+\frac{r}{\rho(x_0)}\right)^{N\cdot\frac{N_0}{N_0+1}}
\left(1+\frac{2^{k+1}r}{\rho(x_0)}\right)^{-N}|f(y)|\,dy.
\end{split}
\end{equation*}
Following along the same lines as that of Theorem $\ref{mainthm:1}$, we are able to show that
\begin{equation*}
\begin{split}
J_3&\leq C\|f\|_{L^{p,\kappa}_{\rho,\theta}(w)}\left(1+\frac{r}{\rho(x_0)}\right)^{\mu}\left(1+\frac{r}{\rho(x_0)}\right)^{N\cdot\frac{N_0}{N_0+1}}\\
&\quad\times\sum_{k=1}^\infty\frac{w(B)^{{(1-\kappa)}/p}}{w(2^{k+1}B)^{{(1-\kappa)}/p}}
\left(1+\frac{2^{k+1}r}{\rho(x_0)}\right)^{-N+\theta+\theta'}\\
&\leq C\|f\|_{L^{p,\kappa}_{\rho,\theta}(w)}\left(1+\frac{r}{\rho(x_0)}\right)^{\mu+N\cdot\frac{N_0}{N_0+1}}\\
&\quad\times\sum_{k=1}^\infty\left(\frac{|B|}{|2^{k+1}B|}\right)^{\delta{(1-\kappa)}/p}
\left(1+\frac{2^{k+1}r}{\rho(x_0)}\right)^{-N+\theta+\theta'+\eta{(1-\kappa)}/p}.
\end{split}
\end{equation*}
The last inequality is obtained by using \eqref{compare}. For any $x\in B(x_0,r)$ and any positive integer $N$, similar to the proof of \eqref{T} and \eqref{Tf2}, we can also deduce that
\begin{equation}\label{zeta}
\begin{split}
\zeta(x)
&=\int_{(2B)^c}\big|b(y)-b_B\big|\big|\mathcal K(x,y)f(y)\big|\,dy\\
&\leq C_N\int_{(2B)^c}\big|b(y)-b_B\big|\bigg(1+\frac{|x-y|}{\rho(x)}\bigg)^{-N}\frac{1}{|x-y|^d}\cdot|f(y)|\,dy\\
&\leq C_{N,d}\sum_{k=1}^\infty\int_{2^{k+1}B\backslash 2^k B}\big|b(y)-b_B\big|\bigg(1+\frac{|x_0-y|}{\rho(x)}\bigg)^{-N}\frac{1}{|x_0-y|^d}\cdot|f(y)|\,dy\\
&\leq C_{N,d}\sum_{k=1}^\infty\frac{1}{|2^{k+1}B|}\int_{2^{k+1}B\backslash 2^k B}\big|b(y)-b_B\big|\bigg(1+\frac{2^kr}{\rho(x)}\bigg)^{-N}|f(y)|\,dy\\
&\leq C\sum_{k=1}^\infty\frac{1}{|2^{k+1}B|}\int_{2^{k+1}B}\big|b(y)-b_B\big|\left(1+\frac{r}{\rho(x_0)}\right)^{N\cdot\frac{N_0}{N_0+1}}
\left(1+\frac{2^{k+1}r}{\rho(x_0)}\right)^{-N}|f(y)|\,dy,
\end{split}
\end{equation}
where in the last inequality we have used \eqref{com2} in Lemma \ref{N0}. Hence, by the above pointwise estimate for $\zeta(x)$,
\begin{equation*}
\begin{split}
J_4&\leq C\cdot w(B)^{{(1-\kappa)}/p}\sum_{k=1}^\infty\left(1+\frac{r}{\rho(x_0)}\right)^{N\cdot\frac{N_0}{N_0+1}}
\left(1+\frac{2^{k+1}r}{\rho(x_0)}\right)^{-N}\\
&\times\frac{1}{|2^{k+1}B|}\int_{2^{k+1}B}\big|b(y)-b_B\big|\big|f(y)\big|\,dy.
\end{split}
\end{equation*}
Moreover, for each integer $k\geq 1$,
\begin{align}\label{est1}
&\frac{1}{|2^{k+1}B|}\int_{2^{k+1}B}\big|b(y)-b_B\big|\big|f(y)\big|\,dy\notag\\
&\leq\frac{1}{|2^{k+1}B|}\int_{2^{k+1}B}\big|b(y)-b_{2^{k+1}B}\big|\big|f(y)\big|\,dy\notag\\
&+\frac{1}{|2^{k+1}B|}\int_{2^{k+1}B}\big|b_{2^{k+1}B}-b_B\big|\big|f(y)\big|\,dy.
\end{align}
By using H\"older's inequality, the first term of the expression \eqref{est1} is bounded by
\begin{equation*}
\begin{split}
&\frac{1}{|2^{k+1}B|}\bigg(\int_{2^{k+1}B}\big|f(y)\big|^pw(y)\,dy\bigg)^{1/p}
\bigg(\int_{2^{k+1}B}\big|b(y)-b_{2^{k+1}B}\big|^{p'}w(y)^{-{p'}/p}\,dy\bigg)^{1/{p'}}\\
&\leq C\|f\|_{L^{p,\kappa}_{\rho,\theta}(w)}\cdot\frac{w(2^{k+1}B)^{{\kappa}/p}}{|2^{k+1}B|}\left(1+\frac{2^{k+1}r}{\rho(x_0)}\right)^{\theta}
\bigg(\int_{2^{k+1}B}\big|b(y)-b_{2^{k+1}B}\big|^{p'}w(y)^{-{p'}/p}\,dy\bigg)^{1/{p'}}.
\end{split}
\end{equation*}
Since $w\in A^{\rho,\theta'}_p$ with $0<\theta'<\infty$ and $1<p<\infty$, then by the definition of $A^{\rho,\theta'}_p$, it can be easily shown that $w\in A^{\rho,\theta'}_p$ if and only if $w^{-{p'}/p}\in A^{\rho,\theta'}_{p'}$, where $1/p+1/{p'}=1$ (see \cite{tang}). If we denote $v=w^{-{p'}/p}$, then $v\in A^{\rho,\theta'}_{p'}$. This fact together with Lemma \ref{BMO1} implies
\begin{equation*}
\begin{split}
&\bigg(\int_{2^{k+1}B}\big|b(y)-b_{2^{k+1}B}\big|^{p'}v(y)\,dy\bigg)^{1/{p'}}\\
&\leq C_{b}\cdot v(2^{k+1}B)^{1/{p'}}\left(1+\frac{2^{k+1}r}{\rho(x_0)}\right)^{\mu}\\
&=C_b\cdot\bigg(\int_{2^{k+1}B}w(y)^{-{p'}/p}\,dy\bigg)^{1/{p'}}\left(1+\frac{2^{k+1}r}{\rho(x_0)}\right)^{\mu}\\
&\leq C_{b,w}\cdot\frac{|2^{k+1}B|}{w(2^{k+1}B)^{1/p}}\left(1+\frac{2^{k+1}r}{\rho(x_0)}\right)^{\theta'}
\left(1+\frac{2^{k+1}r}{\rho(x_0)}\right)^{\mu}.
\end{split}
\end{equation*}
Therefore, the first term of the expression \eqref{est1} can be bounded by a constant times
\begin{equation*}
\frac{w(2^{k+1}B)^{{\kappa}/p}}{w(2^{k+1}B)^{1/p}}\left(1+\frac{2^{k+1}r}{\rho(x_0)}\right)^{\theta+\theta'+\mu}.
\end{equation*}
Since $b\in \mathrm{BMO}_{\rho,\theta''}(\mathbb R^d)$ with $0<\theta''<\infty$, then by Lemma \ref{BMO2}, H\"older's inequality and the $A^{\rho,\theta'}_p$ condition on $w$, the latter term of the expression \eqref{est1} can be estimated by
\begin{equation*}
\begin{split}
&C_b(k+1)\left(1+\frac{2^{k+1}r}{\rho(x_0)}\right)^{\theta''}\frac{1}{|2^{k+1}B|}\int_{2^{k+1}B}\big|f(y)\big|\,dy\\
&\leq C_b(k+1)\left(1+\frac{2^{k+1}r}{\rho(x_0)}\right)^{\theta''}\frac{1}{|2^{k+1}B|}\bigg(\int_{2^{k+1}B}\big|f(y)\big|^pw(y)\,dy\bigg)^{1/p}
\bigg(\int_{2^{k+1}B}w(y)^{-{p'}/p}\,dy\bigg)^{1/{p'}}\\
&\leq C\|f\|_{L^{p,\kappa}_{\rho,\theta}(w)}\cdot(k+1)\frac{w(2^{k+1}B)^{{\kappa}/p}}{w(2^{k+1}B)^{1/p}}
\left(1+\frac{2^{k+1}r}{\rho(x_0)}\right)^{\theta+\theta'+\theta''}.
\end{split}
\end{equation*}
Consequently,
\begin{align}\label{est2}
&\frac{1}{|2^{k+1}B|}\int_{2^{k+1}B}\big|b(y)-b_B\big|\big|f(y)\big|\,dy\notag\\
&\leq C\|f\|_{L^{p,\kappa}_{\rho,\theta}(w)}\cdot(k+1)\frac{w(2^{k+1}B)^{{\kappa}/p}}{w(2^{k+1}B)^{1/p}}
\left(1+\frac{2^{k+1}r}{\rho(x_0)}\right)^{\theta+\theta'+\theta''+\mu}.
\end{align}
Thus, in view of \eqref{est2},
\begin{equation*}
\begin{split}
J_4&\leq C\|f\|_{L^{p,\kappa}_{\rho,\theta}(w)}\cdot w(B)^{{(1-\kappa)}/p}\sum_{k=1}^\infty(k+1)\left(1+\frac{r}{\rho(x_0)}\right)^{N\cdot\frac{N_0}{N_0+1}}
\left(1+\frac{2^{k+1}r}{\rho(x_0)}\right)^{-N}\\
&\times\frac{1}{w(2^{k+1}B)^{{(1-\kappa)}/p}}\left(1+\frac{2^{k+1}r}{\rho(x_0)}\right)^{\theta+\theta'+\theta''+\mu}\\
&=C\|f\|_{L^{p,\kappa}_{\rho,\theta}(w)}
\left(1+\frac{r}{\rho(x_0)}\right)^{N\cdot\frac{N_0}{N_0+1}}\sum_{k=1}^\infty(k+1)\frac{w(B)^{{(1-\kappa)}/p}}{w(2^{k+1}B)^{{(1-\kappa)}/p}}
\left(1+\frac{2^{k+1}r}{\rho(x_0)}\right)^{-N+\theta+\theta'+\theta''+\mu}.
\end{split}
\end{equation*}
Notice that $w\in A^{\rho,\theta'}_p$ with $0<\theta'<\infty$. A further application of \eqref{compare} yields
\begin{equation*}
\begin{split}
J_4&\leq C\|f\|_{L^{p,\kappa}_{\rho,\theta}(w)}
\left(1+\frac{r}{\rho(x_0)}\right)^{N\cdot\frac{N_0}{N_0+1}}\sum_{k=1}^\infty(k+1)\left(\frac{|B|}{|2^{k+1}B|}\right)^{\delta{(1-\kappa)}/p}\\
&\quad\times\left(1+\frac{2^{k+1}r}{\rho(x_0)}\right)^{\eta{(1-\kappa)}/p}\left(1+\frac{2^{k+1}r}{\rho(x_0)}\right)^{-N+\theta+\theta'+\theta''+\mu}\\
&=C\|f\|_{L^{p,\kappa}_{\rho,\theta}(w)}
\left(1+\frac{r}{\rho(x_0)}\right)^{N\cdot\frac{N_0}{N_0+1}}\sum_{k=1}^\infty(k+1)\left(\frac{|B|}{|2^{k+1}B|}\right)^{\delta{(1-\kappa)}/p}\\
&\quad\times\left(1+\frac{2^{k+1}r}{\rho(x_0)}\right)^{-N+\theta+\theta'+\theta''+\mu+\eta{(1-\kappa)}/p}.
\end{split}
\end{equation*}
Combining the above estimates for $J_3$ and $J_4$, we get
\begin{equation*}
\begin{split}
J_2\leq J_3+J_4&\leq C\|f\|_{L^{p,\kappa}_{\rho,\theta}(w)}\left(1+\frac{r}{\rho(x_0)}\right)^{\mu+N\cdot\frac{N_0}{N_0+1}}\\
&\times\sum_{k=1}^\infty(k+1)\left(\frac{|B|}{|2^{k+1}B|}\right)^{\delta{(1-\kappa)}/p}
\left(1+\frac{2^{k+1}r}{\rho(x_0)}\right)^{-N+\theta+\theta'+\theta''+\mu+\eta{(1-\kappa)}/p}.
\end{split}
\end{equation*}
By choosing $N$ large enough so that $N>\theta+\theta'+\theta''+\mu+\eta{(1-\kappa)}/p$, we thus have
\begin{equation*}
\begin{split}
J_2&\leq C\left(1+\frac{r}{\rho(x_0)}\right)^{\mu+N\cdot\frac{N_0}{N_0+1}}
\sum_{k=1}^\infty(k+1)\left(\frac{|B|}{|2^{k+1}B|}\right)^{\delta{(1-\kappa)}/p}\\
&\leq C\left(1+\frac{r}{\rho(x_0)}\right)^{\mu+N\cdot\frac{N_0}{N_0+1}}.
\end{split}
\end{equation*}
Finally, collecting the above estimates for $J_1,J_2$, and letting $\vartheta=\max\big\{\vartheta',\mu+N\cdot\frac{N_0}{N_0+1}\big\}$, we obtain the desired result \eqref{Main3}. The proof of Theorem \ref{mainthm:3} is finished.
\end{proof}

\begin{proof}[Proof of Theorem $\ref{mainthm:4}$]
We denote by $[b,\mathcal T]$ either $[b,\mathcal R]$ or $[b,\mathcal R^{\ast}]$. We are going to prove that for any given $\lambda>0$ and any given ball $B=B(x_0,r)$ of $\mathbb R^d$, there is some $\vartheta>0$ such that the following inequality
\begin{equation}\label{Main4}
\frac{1}{w(B)^{\kappa}}\cdot w\big(\big\{x\in B:|[b,\mathcal T]f(x)|>\lambda\big\}\big)
\leq C\left(1+\frac{r}{\rho(x_0)}\right)^{\vartheta}
\bigg\|\Phi\bigg(\frac{|f|}{\lambda}\bigg)\bigg\|_{(L\log L)^{1,\kappa}_{\rho,\theta}(w)}
\end{equation}
holds for those functions $f$ such that $\Phi(|f|)\in(L\log L)^{1,\kappa}_{\rho,\theta}(w)$ with some fixed $\theta>0$.Now assume that $w\in A^{\rho,\theta'}_1$ for some $\theta'>0$ and $b\in\mathrm{BMO}_{\rho,\theta''}(\mathbb R^d)$ for some $\theta''>0$. As before, we decompose $f$ as
\begin{equation*}
f=f_1+f_2;\quad f_1=f\cdot\chi_{2B},~f_2=f\cdot\chi_{(2B)^c}.
\end{equation*}
Then for any given $\lambda>0$, by the linearity of $[b,\mathcal T]$, we can write
\begin{equation*}
\begin{split}
&\frac{1}{w(B)^{\kappa}}\cdot w\big(\big\{x\in B:|[b,\mathcal T]f(x)|>\lambda\big\}\big)\\
&\leq\frac{1}{w(B)^{\kappa}}\cdot w\big(\big\{x\in B:|[b,\mathcal T](f_1)(x)|>\lambda/2\big\}\big)\\
&+\frac{1}{w(B)^{\kappa}}\cdot w\big(\big\{x\in B:|[b,\mathcal T](f_2)(x)|>\lambda/2\big\}\big)\\
&:=J'_1+J'_2.
\end{split}
\end{equation*}
Let us first estimate the term $J'_1$. By using Theorem \ref{cweak}, we get
\begin{equation*}
\begin{split}
J'_1&=\frac{1}{w(B)^{\kappa}}\cdot w\big(\big\{x\in B:|[b,\mathcal T](f_1)(x)|>\lambda/2\big\}\big)\\
&\leq C\cdot\frac{1}{w(B)^{\kappa}}\bigg[\int_{\mathbb R^d}\Phi\left(\frac{|f_1(x)|}{\lambda}\right)\cdot w(x)\,dx\bigg]\\
&=C\cdot\frac{1}{w(B)^{\kappa}}\bigg[\int_{2B}\Phi\left(\frac{|f(x)|}{\lambda}\right)\cdot w(x)\,dx\bigg].
\end{split}
\end{equation*}
A further application of \eqref{small} yields
\begin{equation*}
\begin{split}
J'_1&\leq C\cdot\frac{w(2B)}{w(B)^{\kappa}}
\bigg\|\Phi\bigg(\frac{|f|}{\lambda}\bigg)\bigg\|_{L\log L(w),2B}\\
&\leq C\cdot\frac{w(2B)^{\kappa}}{w(B)^{\kappa}}\cdot\left(1+\frac{2r}{\rho(x_0)}\right)^{\theta}
\bigg\|\Phi\bigg(\frac{|f|}{\lambda}\bigg)\bigg\|_{(L\log L)^{1,\kappa}_{\rho,\theta}(w)}\\
&\leq C\cdot\left(1+\frac{2r}{\rho(x_0)}\right)^{\kappa\cdot\theta'}\cdot\left(1+\frac{2r}{\rho(x_0)}\right)^{\theta}
\bigg\|\Phi\bigg(\frac{|f|}{\lambda}\bigg)\bigg\|_{(L\log L)^{1,\kappa}_{\rho,\theta}(w)},
\end{split}
\end{equation*}
where the last inequality is due to \eqref{doubling2}. If we denote $\vartheta'=\kappa\cdot\theta'+\theta$, then
\begin{equation*}
\begin{split}
J'_1&\leq C\cdot\left(1+\frac{2r}{\rho(x_0)}\right)^{\vartheta'}
\bigg\|\Phi\bigg(\frac{|f|}{\lambda}\bigg)\bigg\|_{(L\log L)^{1,\kappa}_{\rho,\theta}(w)}
\leq C\cdot\left(1+\frac{r}{\rho(x_0)}\right)^{\vartheta'}
\bigg\|\Phi\bigg(\frac{|f|}{\lambda}\bigg)\bigg\|_{(L\log L)^{1,\kappa}_{\rho,\theta}(w)}
\end{split}
\end{equation*}
as desired. Next let us deal with the term $J'_2$. Taking into account of \eqref{twoterm}, we can divide it into two parts, namely,
\begin{equation*}
\begin{split}
J'_2&=\frac{1}{w(B)^{\kappa}}\cdot w\big(\big\{x\in B:|[b,\mathcal T](f_2)(x)|>\lambda/2\big\}\big)\\
&\leq\frac{1}{w(B)^{\kappa}}\cdot w\big(\big\{x\in B:|\xi(x)|>\lambda/4\big\}\big)+
\frac{1}{w(B)^{\kappa}}\cdot w\big(\big\{x\in B:|\zeta(x)|>\lambda/4\big\}\big)\\
&:=J'_3+J'_4,
\end{split}
\end{equation*}
where
\begin{equation*}
\begin{split}
\xi(x)&=\big|b(x)-b_B\big|\int_{\mathbb R^d}\big|\mathcal K(x,y)f_2(y)\big|\,dy,\\
\&\quad \zeta(x)&=\int_{\mathbb R^d}\big|b(y)-b_B\big|\big|\mathcal K(x,y)f_2(y)\big|\,dy.
\end{split}
\end{equation*}
Since $b\in\mathrm{BMO}_{\rho,\theta''}(\mathbb R^d)$ for some $\theta''>0$, from the pointwise inequality \eqref{Tf2} and Chebyshev's inequality, we then have
\begin{equation*}
\begin{split}
J'_3&\leq\frac{1}{w(B)^{\kappa}}\cdot\frac{\,4\,}{\lambda}\bigg(\int_B\big|\xi(x)\big|w(x)\,dx\bigg)\\
&\leq C\cdot\frac{1}{w(B)^{\kappa}}\bigg(\int_B\big|b(x)-b_B\big|w(x)\,dx\bigg)\\
&\quad\times\sum_{k=1}^\infty\frac{1}{|2^{k+1}B|}\int_{2^{k+1}B}\left(1+\frac{r}{\rho(x_0)}\right)^{N\cdot\frac{N_0}{N_0+1}}
\left(1+\frac{2^{k+1}r}{\rho(x_0)}\right)^{-N}\frac{|f(y)|}{\lambda}\,dy\\
&\leq C_b\cdot\frac{w(B)}{w(B)^{\kappa}}\left(1+\frac{r}{\rho(x_0)}\right)^{\mu}\\
&\quad\times\sum_{k=1}^\infty\frac{1}{|2^{k+1}B|}\int_{2^{k+1}B}\left(1+\frac{r}{\rho(x_0)}\right)^{N\cdot\frac{N_0}{N_0+1}}
\left(1+\frac{2^{k+1}r}{\rho(x_0)}\right)^{-N}\frac{|f(y)|}{\lambda}\,dy,
\end{split}
\end{equation*}
where in the last inequality we have used \eqref{wang3} in Lemma \ref{BMO1}. Moreover, it follows directly from the condition $A^{\rho,\theta'}_1$ that for each integer $k\geq1$,
\begin{equation*}
\begin{split}
&\frac{1}{|2^{k+1}B|}\int_{2^{k+1}B}\frac{|f(y)|}{\lambda}\,dy\\
&\leq \frac{C_w}{w(2^{k+1}B)}\cdot\underset{y\in 2^{k+1}B}{\mbox{ess\,inf}}\;w(y)
\bigg(\int_{2^{k+1}B}\frac{|f(y)|}{\lambda}\,dy\bigg)\left(1+\frac{2^{k+1}r}{\rho(x_0)}\right)^{\theta'}\\
&\leq \frac{C_w}{w(2^{k+1}B)}\bigg(\int_{2^{k+1}B}\frac{|f(y)|}{\lambda}\cdot w(y)\,dy\bigg)\left(1+\frac{2^{k+1}r}{\rho(x_0)}\right)^{\theta'}.
\end{split}
\end{equation*}
Notice also that trivially
\begin{equation}\label{phit}
t\leq t\cdot(1+\log^+t)=\Phi(t),\quad\hbox{for any} ~~t>0.
\end{equation}
This fact along with \eqref{small} implies that for each integer $k\geq1$,
\begin{equation*}
\begin{split}
&\frac{1}{|2^{k+1}B|}\int_{2^{k+1}B}\frac{|f(y)|}{\lambda}\,dy\\
&\leq \frac{C_w}{w(2^{k+1}B)}\bigg(\int_{2^{k+1}B}\Phi\left(\frac{|f(y)|}{\lambda}\right)\cdot w(y)\,dy\bigg)\left(1+\frac{2^{k+1}r}{\rho(x_0)}\right)^{\theta'}\\
&\leq C\cdot\left(1+\frac{2^{k+1}r}{\rho(x_0)}\right)^{\theta'}
\bigg\|\Phi\bigg(\frac{|f|}{\lambda}\bigg)\bigg\|_{L\log L(w),2^{k+1}B}\\
&\leq C\cdot\frac{1}{w(2^{k+1}B)^{1-\kappa}}
\left(1+\frac{2^{k+1}r}{\rho(x_0)}\right)^{\theta}\left(1+\frac{2^{k+1}r}{\rho(x_0)}\right)^{\theta'}
\bigg\|\Phi\bigg(\frac{|f|}{\lambda}\bigg)\bigg\|_{(L\log L)^{1,\kappa}_{\rho,\theta}(w)}.
\end{split}
\end{equation*}
Consequently,
\begin{equation*}
\begin{split}
J'_3&\leq C\left(1+\frac{r}{\rho(x_0)}\right)^{\mu}
\left(1+\frac{r}{\rho(x_0)}\right)^{N\cdot\frac{N_0}{N_0+1}}\\
&\quad\times\sum_{k=1}^\infty \frac{w(B)^{1-\kappa}}{w(2^{k+1}B)^{1-\kappa}}
\left(1+\frac{2^{k+1}r}{\rho(x_0)}\right)^{-N+\theta+\theta'}
\bigg\|\Phi\bigg(\frac{|f|}{\lambda}\bigg)\bigg\|_{(L\log L)^{1,\kappa}_{\rho,\theta}(w)}.
\end{split}
\end{equation*}
Since $w\in A^{\rho,\theta'}_1$ with $0<\theta'<\infty$, then there exist two positive numbers $\delta',\eta'>0$ such that \eqref{compare} holds. Therefore,
\begin{equation*}
\begin{split}
J'_3&\leq C\left(1+\frac{r}{\rho(x_0)}\right)^{\mu+N\cdot\frac{N_0}{N_0+1}}\\
&\times\sum_{k=1}^\infty \left(\frac{|B|}{|2^{k+1}B|}\right)^{\delta'{(1-\kappa)}}
\left(1+\frac{2^{k+1}r}{\rho(x_0)}\right)^{\eta'{(1-\kappa)}}\left(1+\frac{2^{k+1}r}{\rho(x_0)}\right)^{-N+\theta+\theta'}
\bigg\|\Phi\bigg(\frac{|f|}{\lambda}\bigg)\bigg\|_{(L\log L)^{1,\kappa}_{\rho,\theta}(w)}\\
&=C\left(1+\frac{r}{\rho(x_0)}\right)^{\mu+N\cdot\frac{N_0}{N_0+1}}\\
&\times\sum_{k=1}^\infty \left(\frac{|B|}{|2^{k+1}B|}\right)^{\delta'{(1-\kappa)}}
\left(1+\frac{2^{k+1}r}{\rho(x_0)}\right)^{-N+\theta+\theta'+\eta'{(1-\kappa)}}
\bigg\|\Phi\bigg(\frac{|f|}{\lambda}\bigg)\bigg\|_{(L\log L)^{1,\kappa}_{\rho,\theta}(w)}.
\end{split}
\end{equation*}
On the other hand, it follows from the pointwise inequality \eqref{zeta} and Chebyshev's inequality that
\begin{equation*}
\begin{split}
J'_4&\leq\frac{1}{w(B)^{\kappa}}\cdot\frac{\,4\,}{\lambda}\bigg(\int_B\big|\zeta(x)\big|w(x)\,dx\bigg)\\
&\leq C\cdot\frac{w(B)}{w(B)^{\kappa}}\sum_{k=1}^\infty\left(1+\frac{r}{\rho(x_0)}\right)^{N\cdot\frac{N_0}{N_0+1}}
\left(1+\frac{2^{k+1}r}{\rho(x_0)}\right)^{-N}\\
&\quad\times \frac{1}{|2^{k+1}B|}\int_{2^{k+1}B}\big|b(y)-b_B\big|\frac{|f(y)|}{\lambda}\,dy\\
&\leq C\cdot\frac{w(B)}{w(B)^{\kappa}}\sum_{k=1}^\infty\left(1+\frac{r}{\rho(x_0)}\right)^{N\cdot\frac{N_0}{N_0+1}}
\left(1+\frac{2^{k+1}r}{\rho(x_0)}\right)^{-N}\\
&\quad\times \frac{1}{|2^{k+1}B|}\int_{2^{k+1}B}\big|b(y)-b_B\big|\Phi\left(\frac{|f(y)|}{\lambda}\right)\,dy,
\end{split}
\end{equation*}
where the last inequality follows from \eqref{phit}. Furthermore, by the definition of $A^{\rho,\theta'}_1$, we compute
\begin{equation}\label{est3}
\begin{split}
&\frac{1}{|2^{k+1}B|}\int_{2^{k+1}B}\big|b(y)-b_B\big|\Phi\left(\frac{|f(y)|}{\lambda}\right)\,dy\\
&\leq\frac{C_w}{w(2^{k+1}B)}\cdot\underset{y\in 2^{k+1}B}{\mbox{ess\,inf}}\;w(y)
\int_{2^{k+1}B}\big|b(y)-b_B\big|\Phi\left(\frac{|f(y)|}{\lambda}\right)\,dy\left(1+\frac{2^{k+1}r}{\rho(x_0)}\right)^{\theta'}\\
&\leq\frac{C_w}{w(2^{k+1}B)}\int_{2^{k+1}B}\big|b(y)-b_{2^{k+1}B}\big|\Phi\left(\frac{|f(y)|}{\lambda}\right)w(y)\,dy
\left(1+\frac{2^{k+1}r}{\rho(x_0)}\right)^{\theta'}\\
&+\frac{C_w}{w(2^{k+1}B)}\int_{2^{k+1}B}\big|b_{2^{k+1}B}-b_B\big|\Phi\bigg(\frac{|f(y)|}{\lambda}\bigg)w(y)\,dy
\left(1+\frac{2^{k+1}r}{\rho(x_0)}\right)^{\theta'}.
\end{split}
\end{equation}
By using generalized H\"older inequality \eqref{gholder}, the first term of the expression \eqref{est3} is bounded by
\begin{equation*}
\begin{split}
&C\big\|b-b_{2^{k+1}B}\big\|_{\exp L(w),2^{k+1}B}
\bigg\|\Phi\bigg(\frac{|f|}{\,\lambda\,}\bigg)\bigg\|_{L\log L(w),2^{k+1}B}
\left(1+\frac{2^{k+1}r}{\rho(x_0)}\right)^{\theta'}\\
&\leq C\bigg\|\Phi\bigg(\frac{|f|}{\,\lambda\,}\bigg)\bigg\|_{L\log L(w),2^{k+1}B}
\left(1+\frac{2^{k+1}r}{\rho(x_0)}\right)^{\eta'}\left(1+\frac{2^{k+1}r}{\rho(x_0)}\right)^{\theta'},
\end{split}
\end{equation*}
where in the last inequality we have used the fact that
\begin{equation*}
\big\|b-b_{B}\big\|_{\exp L(w),B}\leq C\left(1+\frac{r}{\rho(x_0)}\right)^{\eta'},\quad \mbox{for any ball }B=B(x_0,r)\subset\mathbb R^n,
\end{equation*}
which is equivalent to the inequality \eqref{wang2} in Lemma \ref{BMO3}. By Lemma \ref{BMO2} and \eqref{small}, the latter term of the expression \eqref{est3} can be estimated by
\begin{equation*}
\begin{split}
& C_b \frac{(k+1)}{w(2^{k+1}B)}\int_{2^{k+1}B}\Phi\bigg(\frac{|f(y)|}{\lambda}\bigg)w(y)\,dy
\left(1+\frac{2^{k+1}r}{\rho(x_0)}\right)^{\theta''}\left(1+\frac{2^{k+1}r}{\rho(x_0)}\right)^{\theta'}\\
&\leq C(k+1)\bigg\|\Phi\bigg(\frac{|f|}{\,\lambda\,}\bigg)\bigg\|_{L\log L(w),2^{k+1}B}
\left(1+\frac{2^{k+1}r}{\rho(x_0)}\right)^{\theta'}\left(1+\frac{2^{k+1}r}{\rho(x_0)}\right)^{\theta''}.
\end{split}
\end{equation*}
Consequently,
\begin{equation*}
\begin{split}
J'_4&\leq C\cdot\frac{w(B)}{w(B)^{\kappa}}\sum_{k=1}^\infty\left(1+\frac{r}{\rho(x_0)}\right)^{N\cdot\frac{N_0}{N_0+1}}
\left(1+\frac{2^{k+1}r}{\rho(x_0)}\right)^{-N}\\
&\quad\times(k+1)\bigg\|\Phi\bigg(\frac{|f|}{\,\lambda\,}\bigg)\bigg\|_{L\log L(w),2^{k+1}B}
\left(1+\frac{2^{k+1}r}{\rho(x_0)}\right)^{\eta'+\theta'+\theta''}\\
&\leq C\left(1+\frac{r}{\rho(x_0)}\right)^{N\cdot\frac{N_0}{N_0+1}}\sum_{k=1}^\infty(k+1)\frac{w(B)^{{(1-\kappa)}}}{w(2^{k+1}B)^{{(1-\kappa)}}}
\bigg\|\Phi\bigg(\frac{|f|}{\lambda}\bigg)\bigg\|_{(L\log L)^{1,\kappa}_{\rho,\theta}(w)}\\
\end{split}
\end{equation*}
\begin{equation*}
\begin{split}
&\times\left(1+\frac{2^{k+1}r}{\rho(x_0)}\right)^{-N+\eta'+\theta+\theta'+\theta''}\\
&\leq C\left(1+\frac{r}{\rho(x_0)}\right)^{N\cdot\frac{N_0}{N_0+1}}\sum_{k=1}^\infty(k+1)\left(\frac{|B|}{|2^{k+1}B|}\right)^{\delta'{(1-\kappa)}}\\
&\times\left(1+\frac{2^{k+1}r}{\rho(x_0)}\right)^{-N+\eta'+\theta+\theta'+\theta''+\eta'(1-\kappa)}
\bigg\|\Phi\bigg(\frac{|f|}{\lambda}\bigg)\bigg\|_{(L\log L)^{1,\kappa}_{\rho,\theta}(w)}.
\end{split}
\end{equation*}
Hence, combining the above estimates for $J'_3$ and $J'_4$, we have
\begin{equation*}
\begin{split}
J'_2&\leq J'_3+J'_4\leq C\left(1+\frac{r}{\rho(x_0)}\right)^{\mu+N\cdot\frac{N_0}{N_0+1}}
\sum_{k=1}^\infty(k+1)\left(\frac{|B|}{|2^{k+1}B|}\right)^{\delta'{(1-\kappa)}}\\
&\times\left(1+\frac{2^{k+1}r}{\rho(x_0)}\right)^{-N+\eta'+\theta+\theta'+\theta''+\eta'{(1-\kappa)}}
\bigg\|\Phi\bigg(\frac{|f|}{\lambda}\bigg)\bigg\|_{(L\log L)^{1,\kappa}_{\rho,\theta}(w)}.
\end{split}
\end{equation*}
Now $N$ can be chosen sufficiently large such that $N>\eta'+\theta+\theta'+\theta''+\eta'(1-\kappa)$, and hence the above series is convergent. Finally,
\begin{equation*}
\begin{split}
J'_2&\leq C\left(1+\frac{r}{\rho(x_0)}\right)^{\mu+N\cdot\frac{N_0}{N_0+1}}\sum_{k=1}^\infty(k+1)\left(\frac{|B|}{|2^{k+1}B|}\right)^{\delta'{(1-\kappa)}}
\bigg\|\Phi\bigg(\frac{|f|}{\lambda}\bigg)\bigg\|_{(L\log L)^{1,\kappa}_{\rho,\theta}(w)}\\
&\leq C\left(1+\frac{r}{\rho(x_0)}\right)^{\mu+N\cdot\frac{N_0}{N_0+1}}
\bigg\|\Phi\bigg(\frac{|f|}{\lambda}\bigg)\bigg\|_{(L\log L)^{1,\kappa}_{\rho,\theta}(w)}.
\end{split}
\end{equation*}
Fix this $N$ and set $\vartheta=\max\big\{\vartheta',\mu+N\cdot\frac{N_0}{N_0+1}\big\}$. Thus, combining the above estimates for $J'_1$ and $J'_2$, the inequality \eqref{Main4} is proved and then the proof of Theorem \ref{mainthm:4} is finished.
\end{proof}
The higher order commutators formed by a $\mathrm{BMO}_{\rho,\infty}(\mathbb R^d)$ function $b$ and the operators $\mathcal R$ and its adjoint $\mathcal R^{\ast}$ are usually defined by
\begin{equation}
\begin{cases}
[b,\mathcal R]_mf(x):=\displaystyle\int_{\mathbb R^n}\big[b(x)-b(y)\big]^m \mathcal K(x,y)f(y)\,dy,\quad x\in\mathbb R^d;&\\
[b,\mathcal R^{\ast}]_mf(x):=\displaystyle\int_{\mathbb R^n}\big[b(x)-b(y)\big]^m \mathcal K^{\ast}(x,y)f(y)\,dy,\quad x\in\mathbb R^d;&\\
\quad m=1,2,3,\dots.&
\end{cases}
\end{equation}
Let $\mathcal T$ denote $\mathcal R$ or $\mathcal R^{\ast}$. Obviously, $[b,\mathcal T]_1=[b,\mathcal T]$ which is just the linear commutator \eqref{briesz}, and
\begin{equation*}
[b,\mathcal T]_m=\big[b,[b,\mathcal T]_{m-1}\big],\quad m=2,3,\dots.
\end{equation*}
By induction on $m$, we are able to show that the conclusions of Theorems \ref{mainthm:3} and \ref{mainthm:4} also hold for the higher order commutators $[b,\mathcal T]_m$ with $m\geq2$. The details are omitted here.

\begin{thm}
Let $1<p<\infty$, $0<\kappa<1$ and $w\in A^{\rho,\infty}_p$. If $V\in RH_q$ with $q\geq d$, then for any positive integer $m\geq2$, the higher order commutators $[b,\mathcal R]_m$ and $[b,\mathcal R^{\ast}]_m$ are all bounded on $L^{p,\kappa}_{\rho,\infty}(w)$, whenever $b\in\mathrm{BMO}_{\rho,\infty}(\mathbb R^d)$.
\end{thm}

\begin{thm}
Let $p=1$, $0<\kappa<1$ and $w\in A^{\rho,\infty}_1$. If $V\in RH_q$ with $q\geq d$ and $b\in\mathrm{BMO}_{\rho,\infty}(\mathbb R^d)$, then for any given $\lambda>0$ and any given ball $B=B(x_0,r)$ of $\mathbb R^d$, there exist some constants $C>0$ and $\vartheta>0$ such that the following inequalities
\begin{equation*}
\begin{split}
&\frac{1}{w(B)^{\kappa}}\cdot w\big(\big\{x\in B:|[b,\mathcal R]_mf(x)|>\lambda\big\}\big)
\leq C\left(1+\frac{r}{\rho(x_0)}\right)^{\vartheta}
\bigg\|\Phi_m\bigg(\frac{|f|}{\lambda}\bigg)\bigg\|_{(L\log L)^{1,\kappa}_{\rho,\theta}(w)},\\
&\frac{1}{w(B)^{\kappa}}\cdot w\big(\big\{x\in B:|[b,\mathcal R^{\ast}]_mf(x)|>\lambda\big\}\big)
\leq C\left(1+\frac{r}{\rho(x_0)}\right)^{\vartheta}
\bigg\|\Phi_m\bigg(\frac{|f|}{\lambda}\bigg)\bigg\|_{(L\log L)^{1,\kappa}_{\rho,\theta}(w)}
\end{split}
\end{equation*}
hold for those functions $f$ such that $\Phi_m(|f|)\in(L\log L)^{1,\kappa}_{\rho,\theta}(w)$ with some fixed $\theta>0$, where $\Phi_m(t)=t\cdot(1+\log^+t)^m$, $m=2,3,\dots$.
\end{thm}

\section*{Acknowledgment}

The author would like to thank Professor L. Tang for providing the paper \cite{tang}.

\end{document}